\documentclass{amsart}
\usepackage{graphicx}
\usepackage{amssymb,amscd,amsthm,amsxtra}
\usepackage{latexsym}
\usepackage{epsfig}
\usepackage{esint}
\usepackage[colorlinks]{hyperref}
\usepackage[colorlinks]{hyperref}
\AtBeginDocument{
   \hypersetup{
    linkcolor=blue,
    citecolor=blue,
 }
}
\usepackage{cleveref}

\numberwithin{equation}{section}

\newtheorem{theorem}{Theorem}[section]
\newtheorem{lemma}{Lemma}[section]
\newtheorem{proposition}{Proposition}[section]

\newtheorem{definition}[theorem]{Definition}

\newtheorem{remark}{Remark}[section]

\def\ds{\displaystyle}

\begin{document}

\title[Boundary Harnack in a slit domain]{The $C^{1,\alpha}$ boundary Harnack principle in a slit domain and its application to the Signorini problem}
\author{Chilin Zhang}
\address{School of Mathematical Sciences, Fudan University, Shanghai 200433, China}\email{zhangchilin@fudan.edu.cn}
\begin{abstract}
We prove the $C^{2,\alpha}$ regularity of the free boundary in the Signorini problem with variable coefficients. We use a $C^{1,\alpha}$ boundary Harnack inequality in slit domains. The key method is to study a non-standard degenerate elliptic equation and obtain a $C^{1,\alpha}$ Schauder estimate.
\end{abstract}
\maketitle

\section{Introduction}
\subsection{Basic theory in the Signorini problem}
The thin obstacle problem, or the Signorini problem, consists in minimizing the energy
\begin{equation*}
    J(U)=\int_{D}\frac{1}{2}(\nabla U)^{t}A(\nabla U)+U(x)F(x)dx,\quad U\Big|_{\partial D}=g
\end{equation*}
in a domain $D\subseteq\mathbb{R}^{n+1}$ under the further constraint that $U\geq\psi$ on a $n$-dimensional hypersurface $M\subseteq D$ dividing $D$ into $D_{1}$ and $D_{2}$. One can imagine a scenario of a carpet hung over a string. The minimizing solution satisfies
\begin{equation}\label{general Signorini}
    \left\{\begin{aligned}
        &\mathrm{div}(A\nabla U)=F(x)&\mbox{when }U>\psi\mbox{ or }x\notin M,\\
        &\mathrm{div}(A\nabla U)\leq F(x)&\mbox{ in the weak sense}.
    \end{aligned}\right.
\end{equation}
The free boundary is a codimension-$2$ surface defined as follows:
\begin{definition}
    Let $U\geq\psi$ on $M$ and assume that $U$ satisfies \eqref{general Signorini}. The free boundary $\Gamma\subseteq M$ is defined as $\Gamma:=\partial\Big|_{M}\{U=\psi\}$.
\end{definition}

Without loss of generality, we can assume $\psi=0$, by replacing $U$ with $U-\psi$ and also replacing $F(x)$ with $F(x)-\mathrm{div}(A\nabla\psi)$.

The Signorini problem has applications in semipermeable membranes and perpetual American options. See \cite{PSU12} for more information.

The constant coefficient model is called the ``model problem S'' in \cite{PSU12}. It consists of studying the special case where $A=\delta^{ij}$, $F,\psi=0$, $M=\mathbb{R}^{n}=\{x_{n+1}=0\}$. We further assume that $U$ is an even function in the $x_{n+1}$-direction. In this case, \eqref{general Signorini} becomes
\begin{equation*}
    \Delta U=0\quad\mbox{when}\ U>0\ \mbox{or}\ x_{n+1}\neq0.
\end{equation*}
In fact, the $x_{n+1}$-even function $U$ can also be understood as the Caffarelli-Silvestre extension of a $\frac{1}{2}$-harmonic function $U\Big|_{\mathbb{R}^{n}}$ corresponding to the nonlocal operator $(-\Delta)^{1/2}$, see \cite{CS07}.

In \cite{AC04}, Athanasopoulos and Caffarelli showed that the solution $U(x)$ of the constant coefficient model is $C^{1,1/2}$ on both sides of $\mathbb{R}^{n}$, and this is the optimal regularity of solutions.

In \cite{ACS08}, free boundary points are classified according to their blow-up and frequency. Let $x_{0}\in\Gamma$ be a boundary point, the Almgren's frequency function is defined as
\begin{equation*}
    N_{x_{0}}(U,r):=r\frac{\int_{B_{r}(x_{0})}|\nabla U|^{2}}{\int_{\partial B_{r}(x_{0})}U^{2}}.
\end{equation*}
This is an increasing function. If $\ds N=\lim_{r\to0}N_{x_{0}}(U,r)\in(1,2)$, then the only possibility is $N=3/2$. In this situation, up to a subsequence, the rescaled function
\begin{equation*}
    U_{r}(x):=\frac{U(rx+x_{0})}{(\fint_{\partial B_{r}(x_{0})}U^{2})^{1/2}}
\end{equation*}
converges in $C^{1,\alpha}$ sense to a $3/2$-homogeneous function $C_{n}\cdot Re((x_{n}+ix_{n+1})^{3/2})$ near $0$, after a coordinate rotation. We say such a free boundary point is a regular point.

Assume that the origin is a regular point, the free boundary $\Gamma$ is locally a $C^{1,\alpha}$ graph (a $(n-1)$-dimensional submanifold) for some small $\alpha$, as shown in \cite{ACS08}. In \cite{CSS08}, the result is extended to the fractional Laplace case.

In \cite{ACS08}, the authors first showed that the free boundary is Lipschitz near a regular point. If the rescaling $U_{r}$ converges to $C_{n}\cdot Re((x_{n}+ix_{n+1})^{3/2})$, then for every unit vector $\vec{e}$ with
\begin{equation*}
    \vec{e}\cdot\vec{e}_{n+1}=0,\quad \vec{e}\cdot\vec{e}_{n}>0,
\end{equation*}
and for any $\epsilon$, there is a sufficiently small $r$ (for simplicity, let $r=1$), so that
\begin{equation*}
    \partial_{\vec{e}}U_{1}(x)\geq C_{n}\partial_{\vec{e}}Re((x_{n}+ix_{n+1})^{3/2})-\epsilon\mbox{ in }B_{1}.
\end{equation*}
Notice that $\partial_{\vec{e}}U_{1}(x)$ vanishes on $\{x_{n+1}=0,U(x)=0\}$ and is harmonic elsewhere. The authors then showed that for vectors $\vec{e}$ satisfying
\begin{equation*}
    \vec{e}\cdot\vec{e}_{n+1}=0,\quad \vec{e}\cdot\vec{e}_{n}\geq\delta>0
\end{equation*}
for some sufficiently small $\delta$, it holds that
\begin{equation*}
    \partial_{\vec{e}}U_{1}(x)\geq0\mbox{ in }B_{1/2}.
\end{equation*}
Therefore, the free boundary is a Lipschitz graph
\begin{equation*}
    \Gamma=\{(x_{1},\cdots,x_{n},0):x_{n}=\gamma(x_{1},\cdots,x_{n-1})\}
\end{equation*}
near the origin.

After that, the authors applied the boundary Harnack principle to the ratios $\partial_{i}U/\partial_{n}U(i<n)$ to show the $C^{1,\alpha}$-regularity of the free boundary. More precisely, they first straightened the free boundary using the coordinate change
\begin{equation*}
    (x_{1},\cdots,x_{n+1})\rightarrow(x_{1},\cdots,x_{n-1},x_{n}-\gamma(x_{1},\cdots,x_{n-1}),x_{n+1}),
\end{equation*}
and then applied the following Lipschitz coordinate change to open the ``slit''
\begin{equation*}
    (x_{1},\cdots,x_{n-1},\rho\cos{\theta},\rho\sin{\theta})\rightarrow(x_{1},\cdots,x_{n-1},\rho\cos{\frac{\theta}{2}},\rho\sin{\frac{\theta}{2}}).
\end{equation*}
The original partial derivatives $u_{i}=\partial_{i}U$ in the original coordinate satisfy the uniformly elliptic equation:
\begin{equation*}
    \mathrm{div}(A\nabla u_{i})=0,\quad\lambda I\leq A\leq\Lambda I
\end{equation*}
in the half space $\mathbb{R}^{n+1}_{+}$ and vanish at the boundary. By a version of the boundary Harnack principle in \cite{CFMS81}, the authors proved that the ratios $\partial_{i}U/\partial_{n}U(i<n)$ are $C^{\alpha}$.

For more general Signorini problems \eqref{general Signorini} with variable coefficients, it was shown in \cite{G09} that when $M,\psi$ are $C^{1,\beta}$ for some $\beta>1/2$, $A$ is $C^{1,\gamma}$ for all $\gamma>0$ and $F$ is just H\"older continuous, then $U$ is $C^{1,1/2}$ on both sides of $M$. Later, in \cite{GSVG14,KRS16,RS17}, the optimal regularity of $U$ was improved subsequently to Lipschitz, $W^{1,p}$ and $C^{\alpha}$ coefficients. 

The regular free boundary points are also defined using the Almgren' frequency.
\begin{definition}
    Let $U\geq\psi=0$ on $M$ be a solution of \eqref{general Signorini}. We say that a free boundary point $x_{0}\in\Gamma$ is regular, if the Almgren's frequency function
    \begin{equation*}
        N_{x_{0}}(U,r):=r\frac{\int_{B_{r}(x_{0})}|\nabla U|^{2}}{\int_{\partial B_{r}(x_{0})}U^{2}}
    \end{equation*}
    converges to $3/2$ when $r\to0$.
\end{definition}

In the variable coefficient case, the $C^{1,\alpha}$ regularity of the free boundary $\Gamma$ near a regular point was obtained in \cite{GPSVG16,KRS17a,RS17} under the regularity assumptions above.

\subsection{Higher regularity of the free boundary.}
Once we know that $\Gamma\in C^{1,\alpha}$, we can study the higher regularity of $\Gamma$. The best result is obtained in \cite{KPS15} for the constant coefficients model and in \cite{KRS17b} for analytic coefficients, both showing that $\Gamma\in C^{\omega}$ (an analytic codimension-$2$ surface) near a regular point using the partial hodograph-Legendre transformation.

One can also use the method of higher order boundary Harnack principle to study the regularity of the free boundary. For example, it is shown in \cite{DS14} that $\Gamma\in C^{\infty}$ near a regular point in the ``model problem S''.

In this paper, we allow variable coefficients and use a $C^{1,\alpha}$ boundary Harnack principle to show that $\Gamma\in C^{2,\alpha}$ near a regular point. The main difficulty is that the blow-up limit differs along the free boundary $\Gamma$ due to the changing coefficient matrix. This will be discussed soon in the next subsection.

We assume $D=B_{1}$, $M=\mathbb{R}^{n}=\{x_{n+1}=0\}$ and $\psi=0$. Besides, $U$ is assumed to be $x_{n+1}$-even, meaning
\begin{equation*}
U(x_{1},\cdots,x_{n},x_{n+1})=U(x_{1},\cdots,x_{n},-x_{n+1}),
\end{equation*}
and the symmetric matrix $A$, accordingly, needs to be an $x_{n+1}$-odd function in entries $A^{ij},A^{ji}$ for $i\leq n,j=n+1$, and be an $x_{n+1}$-even function in all other entries.

In the discussion in the remaining part of this paper, we assume that $A$ satisfies the following properties. These properties are to be used in the proof of Theorem~\ref{application}, particularly in \eqref{eq. additional explanation of fm being tangential}.
\begin{itemize}
    \item[(a)] $\lambda I\leq A\leq\Lambda I$ and $[A]_{C^{\alpha}}$ (for some $\alpha<1/2$) is bounded,
    \item[(b)] $A^{ij}=A^{ji}=0$ for $i\leq n,j=(n+1)$,
    \item[(c)] $\partial_{i}A$'s are $C^{\alpha}(\mathbb{R}^{n+1})$ for all $i\leq n$, and $\partial_{i}A^{n+1,n+1}\equiv0$ for $i\leq n$.
    \item[(d)] $A$ is even in the $x_{n+1}$ direction.
\end{itemize}
\begin{remark}
    By assuming $\alpha<1/2$, the regularity of the free boundary $\Gamma$ in Theorem~\ref{application} is $C^{2,\alpha}$ for $\alpha<1/2$. The reason why we are assuming $\alpha<1/2$ will be explained in Remark~\ref{rmk. obstruction of alpha}, see also \eqref{eq. obstruction reason}. The author believes that if $F(x)\equiv0$ (or vanishes on $\{x_{n}=0\}$ with some sufficiently large vanishing order) in Theorem~\ref{application}, then the assumption $\alpha<1/2$ can be replaced by $\alpha\in(0,1)$.
\end{remark}
\begin{remark}
    If the assumptions (a)-(d) hold, then we can assume that the solution $U$ is an even function in the $x_{n+1}$-variable. Otherwise, we replace $U$ with
\begin{equation*}
    \bar{U}(x_{1},\cdots,x_{n+1}):=\frac{1}{2}U(x_{1},\cdots,x_{n},x_{n+1})+\frac{1}{2}U(x_{1},\cdots,x_{n},-x_{n+1}).
\end{equation*}
Then $\bar{U}$ still satisfies the same equation, and the free boundary is unchanged.
\end{remark}

Our description of $C^{2,\alpha}$ regularity of the free boundary is as follows:
\begin{theorem}\label{application}
    Assume that the matrix $A$ satisfies the even conditions (a)-(d), $F(x)\in W^{1,\infty}$, and an $x_{n+1}$-even solution $U$ solves the thin obstacle problem
\begin{equation}\label{U in y abuse}
    \left\{\begin{aligned}
        &\mathrm{div}(A\nabla U)=F(x)&\mbox{when }U>0\mbox{ or }x_{n}\neq0,\\
        &\mathrm{div}(A\nabla U)\leq F(x)&\mbox{ in the weak sense}.
    \end{aligned}\right.
\end{equation}
If $0$ is a regular boundary point, then the free boundary $\Gamma\subseteq\mathbb{R}^{n}$ is a $C^{2,\alpha}$ graph near $0$.
\end{theorem}

We can perform a scaling, replacing $U(x)$ by $R^{-3/2}U(Rx)$ near the origin. This keeps the blow-up of $U(x)$ at the regular boundary invariant, but we can assume $\|D^{i}A\|_{C^{\alpha}}(i\leq n)$ and $[F]_{W^{1,\infty}}$ are small by sending $R\to0$.

Up to a rotation, let us assume that the free boundary $\Gamma$ is a graph in $(x_{1},\cdots,x_{n-1})$ variables near the origin, meaning that there exists $\gamma(x_{1},\cdots,x_{n-1})$ so that
\begin{equation*}
    \Gamma=\{(x_{1},\cdots,x_{n-1},\gamma(x_{1},\cdots,x_{n-1}),0):(x_{1},\cdots,x_{n-1})\in\mathbb{R}^{n-1}\cap B_{1}\}.
\end{equation*}

The strategy is to show $w_{i}=U_{x_{i}}/U_{x_{n}}(i<n)$ is $C^{1,\alpha}$ on $\Gamma$. In fact, by the implicit function, this value represents the gradient of the level set $\mathbb{R}^{n}\cap\{U=h\}$ for each $h>0$ when written as a graph of $x^{T}$. If $w_{i}\Big|_{\Gamma}\in C^{1,\alpha}$, then so do $\partial_{i}\gamma$, meaning that $\Gamma$ is a $C^{2,\alpha}$ graph.

We perform a coordinate change
\begin{equation}\label{ytox abuse}
    x\rightsquigarrow x-\gamma(x_{1},\cdots,x_{n-1})\vec{e}_{n}
\end{equation}
which straightens the boundary $\Gamma$ to $\mathbb{R}^{n-1}$ (it will be made more clear in \eqref{ytox}). Then, we establish a $C^{1,\alpha}$ boundary Harnack principle on the straightened boundary.

We denote the ``straightened-slit'' as $S$, more precisely,
\begin{equation}
    S=\{x=(x_{1},...,x_{n+1}): x_{n}<0,x_{n+1}=0\}.
\end{equation}
It corresponds to the region $\{x_{n}\leq\gamma(x_{1},...,x_{n-1})\}\cap\mathbb{R}^{n}=\{U=0\}\cap\mathbb{R}^{n}$ before the coordinate change \eqref{ytox abuse}.

For each point $(x_{1},\cdots,x_{n+1})$ after the coordinate change, we abbreviate its first $(n-1)$ indices as $x^{T}$, and last two indices $(x_{n},x_{n+1})=x^{\perp}$.

As the blow-up scaling
\begin{equation*}
    U_{r}(x):=\frac{U(rx+x_{0})}{(\fint_{\partial B_{r}(x_{0})}U^{2})^{1/2}}
\end{equation*}
converges as $r\to0$ in $C^{1,\alpha}$ sense to a $3/2$-homogeneous solution like the previously mentioned
\begin{equation*}
    Re((x_{n}+i x_{n+1})^{3/2}),
\end{equation*}
we see that for $i\leq n$, $U_{x_{i}}$ converges to a multiple of a $1/2$-homogeneous solution like
\begin{equation*}
    Re((x_{n}+i x_{n+1})^{1/2})
\end{equation*}
in $C^{\alpha}$ sense near the regular point. In fact, the blow-up limit of $U_{x_{i}}$ is more complicated, in the sense that the expression depends on the coefficient matrix $A$. The dependence is presented below.

\subsection{Homogeneous solutions} Assume that after straightening the boundary $\Gamma$, which is turned into $\mathbb{R}^{n-1}$, we have a corresponding matrix $A$ (will be explicit in \eqref{equationu}). Let $\bar{A}=A(0)$ so that $\bar{A}^{i,n+1}=\bar{A}^{n+1,i}=0$ for all $i\leq n$, then the positive $1/2$-homogeneous solution of
\begin{equation}
    \mathrm{div}(\bar{A}\nabla\bar{u})=0\ \mbox{in}\ \mathbb{R}^{n+1}\setminus S,\quad\bar{u}\Big|_{S}=0
\end{equation}
is given up to a constant multiple by
\begin{equation}\label{diffenent blow up}
\bar{u}_{\bar{A}}=\sqrt{\frac{x_{n}+\sqrt{x_{n}^{2}+\kappa^{2}x_{n+1}^{2}}}{2}},\quad\kappa=\kappa(\bar{A})=(\frac{\bar{A}^{nn}}{\bar{A}^{n+1,n+1}})^{1/2}.
\end{equation}
In particular, when $\bar{A}=\delta^{ij}$, then $\kappa=1$, and we set
\begin{equation}\label{rho and xi}
\rho:=\sqrt{x_{n}^{2}+x_{n+1}^{2}}=|x^{\perp}|,\quad\xi:=\bar{u}_{I}=\sqrt{\frac{x_{n}+\rho}{2}}=Re\Big((x_{n}+i x_{n+1})^{1/2}\Big).
\end{equation}
Simple computation yields that $0\leq\xi\leq\sqrt{\rho}$ and $\ds|\nabla\xi|=\frac{1}{2\sqrt{\rho}}$. We will use these new coordinates quite often.

It suffices to study the special case $\bar{A}=\delta^{ij}$ since homogeneous solutions differ only by a stretching:
\begin{equation*}
    \bar{u}_{\bar{A}}(x_{n},x_{n+1})=\xi(x_{n},\kappa x_{n+1}).
\end{equation*}
It's worthwhile to mention that when $A$ is uniformly elliptic, then $\kappa(A)$ is bounded from above and below when we move the center point elsewhere, and there is a uniform constant $C(\lambda,\Lambda)$ so that for any $x$,
\begin{equation*}
    C(\lambda,\Lambda)^{-1}\leq\frac{\bar{u}_{\bar{A}_{1}}}{\bar{u}_{\bar{A}_{2}}}(x)\leq C(\lambda,\Lambda).
\end{equation*}
However, the ratio $\ds\frac{\bar{u}_{\bar{A}_{1}}}{\bar{u}_{\bar{A}_{2}}}(x)$ is not continuous when $\kappa(A_{1})\neq\kappa(A_{2})$. It depends only on the angle $arg(x_{n}+i x_{n+1})$, so when $|x^{\perp}|\to0$, the oscillation of $\ds\frac{\bar{u}_{\bar{A}_{1}}}{\bar{u}_{\bar{A}_{2}}}(x)$ remains large. Besides, one can check that the ratio $\ds\frac{\bar{u}_{\bar{A}_{1}}}{\bar{u}_{\bar{A}_{2}}}(x)$ is Lipschitz with respect to the angle $arg(x_{n}+i x_{n+1})$.

\subsection{\texorpdfstring{$C^{1,\alpha}$}{Lg} boundary Harnack principle} In this paper, we will sometimes use a path distance to describe the metric space in $\mathbb{R}^{n+1}\setminus S$, that is
\begin{equation}\label{eq. path distance}
    dist(x,y):=\inf\{|\gamma|:\gamma(0)=x,\gamma(1)=y,\gamma(t)\in\mathbb{R}^{n+1}\setminus S\}.
\end{equation}
We describe a kind of H\"older continuity which will be used quite often in this paper. We say a scalar/vector valued function $f(x)$ satisfies property $(\mathcal{F}_{A})$ in $B_{R}\setminus S$ if
\begin{itemize}
    \item[$(\mathcal{F}_{A})$] There exists a constant $C$ and $h(x)$ so that $f(x)=\bar{u}_{A(x^{T})}(x)h(x)$ and
    \begin{equation*}
        \|h\|_{C^{\alpha}(B_{R}\setminus S)}\leq C,\quad\mbox{w.r.t. path distance}\ dist(\cdot,\cdot)\mbox{ as in \eqref{eq. path distance}}.
    \end{equation*}
\end{itemize}
Precisely, $\|h\|_{C^{\alpha}(B_{R}\setminus S)}$ equals the following value:
\begin{equation*}
    \|h\|_{C^{\alpha}(B_{R}\setminus S)}=\sup_{x\in B_{R}\setminus S}|h(x)|+\sup_{x,y\in B_{R}\setminus S}\frac{|h(x)-h(y)|}{dist(x,y)^{\alpha}}.
\end{equation*}
We can also simply write $(\mathcal{F}_{A})$ as $(\mathcal{F})$, provided that the matrix $A$ is known on $\mathbb{R}^{n-1}$.
\begin{remark}
    Besides, if $f(x)$ is $x_{n+1}$-even, then there's no distinction between using the path distance $dist(\cdot,\cdot)$ and using $|x-y|$ to describe regularity. In fact, for an $x_{n+1}$-even function $f(x)$, $\|f\|_{C^{\alpha}(B_{R}\setminus S)}=\|f\|_{C^{\alpha}(B_{R}\cap\{x_{n+1}\geq0\})}$ with respect to the path distance \eqref{eq. path distance}, and $dist(x,y)=|x-y|$ for $x,y\in B_{R}\cap\{x_{n+1}\geq0\}$.
\end{remark}

The $C^{1,\alpha}$ boundary Harnack principle is as follows.
\begin{theorem}\label{HOBH straight}
    Assume that $\lambda I\leq A\leq\Lambda I$ and $A\in C^{\alpha}$($\alpha<1/2$). $\phi_{1},\phi_{2}$ are $L^{\infty}$ functions. $u_{1},u_{2}$ are defined on $B_{1}$, and they satisfy
    \begin{equation}\label{standardx}
        \mathrm{div}(A\cdot\nabla u_{i})=\mathrm{div}(\vec{f}_{i})+\phi_{i}\ \mbox{in}\ B_{1}\setminus S,\quad u_{i}\Big|_{B_{1}'\cap S}=0.
    \end{equation}
    We also assume $u_{2}/\xi\geq c_{0}>0$ in $B_{1}$. If $\vec{f}_{1}$ and $\vec{f}_{2}$ both have property $(\mathcal{F}_{A})$ in $B_{1}$ with $\vec{f}_{i}\cdot\vec{e}_{n+1}\equiv0$, then the ratio $\ds w=\frac{u_{1}}{u_{2}}$ can be continuously extended to $\mathbb{R}^{n-1}$, and $w\Big|_{\mathbb{R}^{n-1}\cap B_{1/2}}$ is $C^{1,\alpha}$. Besides, $u_{1}$ and $u_{2}$ both have property $(\mathcal{F}_{A})$ in $B_{1/2}$.
\end{theorem}

In the next subsection, in order to prove the $C^{1,\alpha}$ boundary Harnack principle, we follow a similar method as in \cite{TTV22}. We derive the equation satisfied by the ratio $w$, which can be turned into a degenerate equation. The $C^{1,\alpha}$ boundary Harnack principle is then a consequence of a Schauder estimate of the ratio $w$.
\subsection{A degenerate equation}
By direct computation, we obtain the equation satisfied by the ratio.
\begin{lemma}\label{ratio}
    If $u_{1}$ and $u_{2}$ satisfy $\mathrm{div}(A\cdot\nabla u_{i})=\mathrm{div}(\vec{f}_{i})+\phi_{i}$, then the ratio $\ds w=\frac{u_{1}}{u_{2}}$ satisfies
    \begin{equation}\label{eq. ratio equation}
        \mathrm{div}(u_{2}^{2}A\nabla w)=\mathrm{div}(u_{2}\vec{f}_{1}-u_{1}\vec{f_{2}})+(\vec{f}_{2}\cdot\nabla u_{1}-\vec{f}_{1}\cdot\nabla u_{2})+(u_{2}\phi_{1}-u_{1}\phi_{2})
    \end{equation}
    as long as $A$ is symmetric.
\end{lemma}
\begin{proof}
    As $u_{1}=u_{2}w$, by expanding $\ds u_{2}\mathrm{div}\Big(A\nabla(u_{2}w)\Big)=u_{2}\mathrm{div}(\vec{f}_{1})+u_{2}\phi_{1}$, we have
    \begin{equation*}
        u_{2}^{2}\mathrm{div}(A\nabla w)+2u_{2}A(\nabla u_{2},\nabla w)=u_{2}[\mathrm{div}(\vec{f}_{1})+\phi_{1}]-u_{2}w\cdot \mathrm{div}(A\nabla u_{2}).
    \end{equation*}
    Its left-hand side is
    \begin{equation*}
        u_{2}^{2}\mathrm{div}(A\nabla w)+A(\nabla u_{2}^{2},\nabla w)=\mathrm{div}(u_{2}^{2}A\nabla w),
    \end{equation*}
    and right-hand side is
    \begin{align*}
        &u_{2}[\mathrm{div}(\vec{f}_{1})+\phi_{1}]-u_{1}[\mathrm{div}(\vec{f}_{2})+\phi_{2}]\\
        =&\mathrm{div}(u_{2}\vec{f}_{1}-u_{1}\vec{f_{2}})+(\vec{f}_{2}\cdot\nabla u_{1}-\vec{f}_{1}\cdot\nabla u_{2})+u_{2}\phi_{1}-u_{1}\phi_{2}.
    \end{align*}
\end{proof}
In the context of $C^{1,\alpha}$ boundary Harnack principle, we can also absorb the term $(\vec{f}_{2}\cdot\nabla u_{1}-\vec{f}_{1}\cdot\nabla u_{2})$ into the divergence term, see Remark~\ref{absorb remark} at the end of section~\ref{only average}.

Assume that $A(0)=\delta^{ij}$ for simplicity, meaning that the blow-up of $u_{2}$ at the origin is $\xi$ (defined in \eqref{rho and xi}). We consider a model degenerate equation:
\begin{equation}\label{main}
    \mathrm{div}(\xi^{2}A\cdot\nabla w)=\mathrm{div}(\xi^{2}\vec{f})+\xi^{2}g\quad\mbox{in}\ B_{1}\setminus S.
\end{equation}
Other terms can also appear on the right-hand side, and we give some explanation in Remark~\ref{absorb remark} and Remark~\ref{xiphi on the RHS}.

The assumption on $w$ is that $w\in L^{2}(B_{1}\setminus S,\rho^{-1}dx)\cap H^{1}_{loc}$, and later we will obtain a Schauder estimate for \eqref{main}. To be precise about the notation, $w\in L^{2}(B_{1}\setminus S,\rho^{-1}dx)$ means that for $\rho$ defined in \eqref{rho and xi},
\begin{equation*}
    \int_{B_{1}\setminus S}\frac{w^{2}}{\rho}dx<\infty.
\end{equation*}
Besides, we say $w\in H^{1}_{loc}$ if for any point $x\in B_{1}\setminus S$, there exists a radius $\epsilon$ such that $B_{\epsilon}(x)\subseteq B_{1}\setminus S$ and $w\in H^{1}(B_{\epsilon}(x))$.

To describe its solution, we realize that if $w$ is an $x_{n+1}$-even, the linearization of $w$ is a non-smooth ``polynomial'' in the variable $(x_{1},...,x_{n},\rho)$. We make the following definition.
\begin{definition}\label{def. "polynomials"}
We refer to ``polynomials'' as elements in the ring $\mathbb{R}[x_{1},...,x_{n},\rho]$. A ``polynomial'' is called linear, if it is in the form
\begin{equation*}
    L=c_{0}+c_{1}x_{1}+\cdots+c_{n}x_{n}+c_{\rho}\rho.
\end{equation*}
\end{definition}
We also define ``pointwise'' $C^{\alpha}$ vector fields, which are the gradients of some linear ``polynomial'' with $O(|x|^{\alpha})$ error.
\begin{definition}\label{def. 1.6}
    A vector field $\vec{f}$ is called ``pointwise'' $C^{\alpha}$ at the origin, if it can be written as $\vec{f}=\vec{f}_{0}+\vec{f}_{\alpha}$, where
\begin{equation*}
\vec{f}_{0}=c_{1}e_{1}+,\cdots+c_{n}\vec{e}_{n}+c_{\rho}\nabla\rho,\quad|\vec{f}_{\alpha}|\leq c_{\alpha}|x|^{\alpha},
\end{equation*}
and we denote its norm as $\|\vec{f}\|_{C^{\alpha}(B_{1},0)}:=|c_{1}|+\cdots+|c_{n}|+|c_{\rho}|+c_{\alpha}$.
\end{definition}
We prove the following pointwise Schauder estimate:
\begin{proposition}\label{gammaschauder}
Assume that $A^{ij}(0)=\delta^{ij}$, $[A^{ij}]_{C^{\alpha}(B_{1}^{+})}\leq\varepsilon_{0}$ and $w\in L^{2}(B_{1}\setminus S,\rho^{-1}dx)\cap H^{1}_{loc}$ is an $x_{n+1}$-even solution of \eqref{main}. Let $\ds p=\frac{n+3}{1-\alpha}$. Given that $\varepsilon_{0}$ is small, there exist a linear ``polynomial'' $L$ and $C=C(n,\alpha)$ such that
\begin{align}\label{average C1alpha}
    &\|L\|_{C^{0,1}(B_{1}\setminus S)}^{2}+\frac{1}{r^{n+2+2\alpha}}\int_{B_{r}\setminus S}\frac{|w-L|^{2}}{\rho}dx\notag\\
    \leq&C\Big\{\|\vec{f}\|_{C^{\alpha}(B_{1},0)}^{2}+\|g\|_{L^{p}(B_{1}\setminus S,\rho\xi^{2}dx)}^{2}+\|w\|_{L^{2}(B_{1}\setminus S,\rho^{-1}dx)}^{2}\Big\}.
\end{align}
\end{proposition}

We can call such a description of $w$ as $C^{1,\alpha}$ in average sense. If $w$ is $C^{1,\alpha}$ in average sense everywhere on $\mathbb{R}^{n-1}$, then its restriction on $\mathbb{R}^{n-1}$ is $C^{1,\alpha}$ in the classical sense.

The $L^{p}$ assumption on $g$ can be weakened, since it will be used only in \eqref{g actual} during the proof, and we just need to assume
\begin{equation}\label{g weaken}
    \frac{1}{r^{(n+2+2\alpha)}}\int_{B_{r}\setminus S}\rho\xi^{4}g^{2}dx\leq C,\quad\forall r\leq1.
\end{equation}

The Schauder estimate for \eqref{main} would be trivial, if the weight $\lambda=\xi^{2}$ were an $A_{2}$-Muckenhoupt weight (see \cite{Fabes2,Fabes3,Fabes1}), which satisfies the condition that for any ball $B$,
\begin{equation*}
    (\frac{1}{|B|}\int_{B}\lambda)\cdot(\frac{1}{|B|}\int_{B}\lambda^{-1})\leq C.
\end{equation*}
Unfortunately, one can verify that the weight $\lambda=\xi^{2}$ is not $A_{2}$-Muckenhoupt. Nevertheless, inspired by \cite{DP23a,DP23b}, we derive the Schauder estimate by working with a weighted Sobolev space.

In order to apply Proposition~\ref{gammaschauder} to the equation in Lemma~\ref{ratio}, we need to show that $u_{2}/\xi$ is $C^{\alpha}$ at the origin, and has a positive lower bound. Therefore, we use a similar method to study a uniform equation
\begin{equation}\label{uniform}
    \mathrm{div}(A\nabla u)=\mathrm{div}(\frac{\vec{f}}{\sqrt{\rho}})\ \mbox{in}\ B_{1}\setminus S,\quad u\Big|_{S}=0
\end{equation}
in trace sense and $\vec{f}=0$ on $\mathbb{R}^{n-1}$. In Theorem~\ref{Holder estimate} we will show that $u$ will have property $(\mathcal{F}_{A})$ if $\vec{f}\in C^{\alpha}$. We also use it to prove a Hopf-type result in Proposition~\ref{hopf}.
\subsection{More comments}
We see from \eqref{diffenent blow up} that different matrices $A$ correspond to different homogeneous solutions $\bar{u}_{A}$, the vertical coordinate change \eqref{ytox abuse} is not suitable in proving $\Gamma\in C^{3,\alpha},\cdots,C^{\omega}$.

If there is a coordinate change different from \eqref{ytox abuse}, so that $\kappa(A)$ is constant along $\mathbb{R}^{n-1}$, then we can take tangential derivatives of $w_{i}$ and obtain higher order regularity of $w_{i}$, perhaps also obtain the analyticity of $\Gamma$ like in \cite{KPS15}\cite{KRS17b} by a bootstrap argument.

We consider that if $A=A_{D}$ or equivalently $A_{O}=0$, and if $A^{n+1,n+1}$ is a constant, then a geodesic flow that starts from $\Gamma$ and its normal vector $\nu\in\mathbb{R}^{n}$ will be the desired coordinate change. To be precise, at each point $x\in\Gamma$, we solve the ODE that
\begin{equation*}
    \nabla_{c'}c'=0,\quad c(0)=x,\quad c'(0)=\nu(x).
\end{equation*}
This locally parameterizes the hyperplane $x_{n+1}=0$ as $\Gamma\times(-\epsilon,\epsilon)$. We then write the equation \eqref{general Signorini} in this new coordinate.

Besides, in order to prove $\Gamma\in C^{\omega}$, we also need to improve Proposition~\ref{gammaschauder}, so that $C^{1,\alpha}$ estimate is in the classical sense rather than in the average sense.

The coordinate change obtained by geodesic flow will be $C^{1,\alpha}$ only when we assume $\Gamma\in C^{2,\alpha}$, so in this paper we still have to use the vertical coordinate change. It seems to be a common phenomenon (like \cite{DS14} and \cite{DS15}) that the proof of $C^{2,\alpha}$ regularity of the free boundary is different from the proof of higher regularity.

\section{A weighted Sobolev space}
In this section we define a weighted $H^{1}$ space, and establish the corresponding weighed energy estimate.
\subsection{Poincar\'e inequality} In this sub-section, we establish a suitable Poincar\'e inequality in Proposition~\ref{thin poincare}. In this paper, we need to introduce a complex coordinate. For each $x=(x^{T},x^{\perp})$ with $x^{\perp}=(x_{n},x_{n+1})$, we write
\begin{equation}\label{complex cordinate}
    \xi+i\eta=(x_{n}+ix_{n+1})^{1/2},
\end{equation}
where $\xi$ here is the same as \eqref{rho and xi}. In the $(x_{1},\cdots,x_{n-1},\xi,\eta)$-coordinate, we first fix the first $(n-1)$ coordinates $(x_{1},\cdots,x_{n-1})$ and obtain a Poincar\'e estimate for the $(\xi,\eta)$-coordinates.
\begin{lemma}
    Assume that $w(\xi,\eta)\in C^{1}_{loc}$ in the two-dimensional half space $\{(\xi,\eta):\xi>0\}$. If $w\Big|_{|x|>r}=0$ for some $r$ (equivalently, $w\Big|_{\xi^{2}+\eta^{2}>r}=0$), then
    \begin{equation}\label{regular poincare}
\int_{\{\xi>0\}}w^{2}d\xi d\eta\leq4\int_{\{\xi>0\}}\xi^{2}|\nabla w|^{2}d\xi d\eta.
\end{equation}
\end{lemma}
\begin{proof}
First, we fix the coordinate $\eta$, then by changing the order of integration,
    \begin{equation*}
\int_{0}^{\infty}|w(\xi,\eta)|d\xi\leq\int_{0}^{\infty}\int_{\xi}^{\infty}|\partial_{\xi}w(h,\eta)|dhd\xi=\int_{0}^{\infty}h|\partial_{\xi}w(h,\eta)|dh.
    \end{equation*}
We now set $\eta$ free and integrate in $\eta$, then
    \begin{equation*}
        \int_{\{\xi>0\}}|w(\xi,\eta)|d\xi d\eta\leq\int_{\{\xi>0\}}\xi|\partial_{\xi}w(\xi,\eta)|d\xi d\eta\leq\int_{\{\xi>0\}}\xi|\nabla w(\xi,\eta)|d\xi d\eta.
    \end{equation*}
We replace $w$ with $w^{2}$, then by Cauchy-Schwartz inequality,
    \begin{equation*}
        4\int_{\{\xi>0\}}w^{2}\int_{\{\xi>0\}}\xi^{2}|\nabla w|^{2}\geq(\int_{\{\xi>0\}}\xi|\nabla w^{2}|)^{2}\geq(\int_{\{\xi>0\}}w^{2})^{2}.
    \end{equation*}
\end{proof}

Consequently, we have the following Poincar\'e inequality for a slit domain.
\begin{proposition}\label{thin poincare}
If a function $w\in C^{1}_{loc}(\mathbb{R}^{n+1}\setminus S)$ is compactly supported, i.e. $w\Big|_{|x|>r}=0$, then
\begin{equation}
    \int_{B_{r}\setminus S}\frac{w^{2}}{\rho}dx\leq4\int_{B_{r}\setminus S}\xi^{2}|\nabla w|^{2}dx.
\end{equation}
\end{proposition}
\begin{proof}
It suffices to show that for each fixed $x^{T}$,
\begin{equation*}
\int_{\mathbb{C}\setminus\mathbb{R}_{-}}\frac{w^{2}}{\rho}dx_{n}dx_{n+1}\leq 4\int_{\mathbb{C}\setminus\mathbb{R}_{-}}\xi^{2}|\nabla^{\perp}w|^{2}dx_{n}dx_{n+1}.
\end{equation*}
Between two coordinates $(x_{n},x_{n+1})$ and $(\xi,\eta)$, we have
\begin{equation*}
    d\xi d\eta=\frac{1}{4\rho}dx_{n}dx_{n+1},\quad|\nabla_{\xi,\eta}w|^{2}d\xi d\eta=|\nabla^{\perp}w|^{2}dx_{n}dx_{n+1}.
\end{equation*}
Now, by writing \eqref{regular poincare} in the $(x_{n},x_{n+1})$-coordinate, we have
\begin{align*}
&\int_{\mathbb{C}\setminus\mathbb{R}_{-}}\frac{w^{2}}{\rho}dx_{n}dx_{n+1}=4\int_{Re(\alpha)>0}w^{2}d\xi d\eta\\
\leq&4\int_{Re(\alpha)>0}\xi^{2}|\nabla_{\xi,\eta}w|^{2}d\xi d\eta=4\int_{\mathbb{C}\setminus\mathbb{R}_{-}}\xi^{2}|\nabla^{\perp}w|^{2}dx_{n}dx_{n+1}.
\end{align*}
\end{proof}
It is then natural for us to define a weighted Sobolev space $H^{1}(B_{r}\setminus S,\xi^{2}dx)$. It is a subspace of $H^{1}_{loc}(B_{r}\setminus S)$ functions which locally have weak derivatives in $B_{r}\setminus S$, such that its weighted energy is bounded.
\begin{definition}
    $H^{1}(B_{r}\setminus S,\xi^{2}dx)$ is the space of $H^{1}_{loc}$ functions $w$ such that
\begin{equation}
    \|w\|_{H^{1}(B_{r}\setminus S,\xi^{2}dx)}^{2}:=\int_{B_{r}\setminus S}\xi^{2}|\nabla w|^{2}dx+\int_{B_{r}\setminus S}\frac{w^{2}}{\rho}dx<\infty.
\end{equation}
Its subspace $H^{1}_{0}(B_{r}\setminus S,\xi^{2}dx)$ is the set of all $w\in H^{1}(B_{r}\setminus S,\xi^{2}dx)$, such that $w=0$ on $(\partial B_{r})\setminus S$ in the trace sense.
\end{definition}

\begin{remark}
    The space $H^{1}_{0}(B_{r}\setminus S,\xi^{2}dx)$ can also be understood as the completion of smooth functions vanishing outside $B_{r}$ with respect to the $H^{1}(B_{r}\setminus S,\xi^{2}dx)$ norm. See Remark~\ref{rmk. density} below.
\end{remark}
\begin{remark}\label{rmk. density}
    For each $w\in H^{1}(B_{r}\setminus S,\xi^{2}dx)$, there exists $w_{i}\in H^{1}(B_{r}\setminus S,\xi^{2}dx)$, so that $w_{k}\in C^{\infty}(B_{r}\setminus S)\cap C^{0}(\overline{B_{r}\setminus S})$ as well, and
\begin{equation*}
    \int_{B_{r}\setminus S}\xi^{2}|\nabla(w-w_{i})|^{2}dx+\int_{B_{r}\setminus S}\frac{(w-w_{i})^{2}}{\rho}dx\to0.
\end{equation*}

To see this, it suffices to just argue with a specific $r=1$. We can use the complex coordinate $(x^{T},\xi,\eta)$ mentioned in \eqref{complex cordinate}. It turns $B_{1}\setminus S$ to $\Omega_{1}=\{(x^{T},\xi,\eta):\xi>0,|x^{T}|^{2}+(\xi^{2}+\eta^{2})^{2}\leq1\}$, and the slit $S$ becomes the plane $\{\xi=0\}$, which is now a $C^{1}$ boundary. The corresponding weighted Sobolev semi-norm $\ds\int_{B_{1}\setminus S}\xi^{2}|\nabla w|^{2}dx$ becomes
\begin{equation*}
    \int_{\Omega_{1}}\Big\{\xi^{2}[(\partial_{\xi}w)^{2}+(\partial_{\eta}w)^{2}]+4\xi^{2}\rho\sum_{i}^{n-1}(\partial_{i}w)^{2}\Big\}dx_{1}\cdots dx_{n-1}d\xi d\eta.
\end{equation*}
In the $(x^{T},\xi,\eta)$ coordinate, let
\begin{equation*}
    L_{h}(x^{T},\xi,\eta)=((1-4h)x^{T},(1-4h)\xi+2h,(1-4h)\eta),\quad h>0
\end{equation*}
be an affine map that maps $\Omega_{1}$ into $\Omega_{1}$. Notice that
\begin{equation*}
    2h\leq dist(x,L(\Omega_{1}))\leq8h,\quad\forall x\in\partial\Omega_{1}.
\end{equation*}
We can then follow the standard method of smooth approximation, that we let $\phi_{h}$ be a smooth mollifier supported in $B_{h}$, and approximate $w$ by the smooth function
\begin{equation*}
    w_{h}(x^{T},\xi,\eta)=(w*\phi_{h})\circ L_{h}(x^{T},\xi,\eta),\quad h\to0.
\end{equation*}
\end{remark}
\begin{remark}
    As a simple consequence of the density of smooth functions inside $H^{1}(B_{r}\setminus S,\xi^{2}dx)$, we see that Proposition~\ref{thin poincare} not only holds for $C^{1}_{loc}$ functions, but also for functions in $H^{1}(\mathbb{R}^{n+1}\setminus S,\xi^{2}dx)$ with compact support.
\end{remark}

\subsection{Caccioppoli inequalities}In this subsection, we develop two Caccioppoli inequalities for the degenerate model \eqref{main}. We write it down again.
\begin{equation}\label{main again}
    \mathrm{div}(\xi^{2}A\cdot\nabla w)=\mathrm{div}(\xi^{2}\vec{f})+\xi^{2}g\quad\mbox{in}\ B_{1}\setminus S.
\end{equation}

First we establish an interior estimate.

\begin{proposition}\label{caccioppoli}
Assume that $\lambda I\leq A\leq\Lambda I$. If $w\in H^{1}_{loc}$ is a solution of \eqref{main again}, then there exists $C=C(\lambda,\Lambda)$ so that
\begin{equation}
    \int_{B_{r/2}\setminus S}\xi^{2}|\nabla w|^{2}dx\leq C\Big\{\int_{B_{r}\setminus S}\frac{w^{2}}{\rho}dx+\int_{B_{r}\setminus S}(\xi^{2}|\vec{f}|^{2}+\rho\xi^{4}g^{2})dx\Big\}.
\end{equation}
\end{proposition}
\begin{proof}
Let $\varphi(x),\eta_{h}(x)$ be two positive smooth functions satisfying
\begin{align*}
    &\varphi\Big|_{|x|\leq r/2}=1,\quad\varphi\Big|_{|x|\geq r}=0,\quad|\nabla\varphi|\leq\frac{C}{r},\\
    &\eta_{h}\Big|_{\xi\geq h}=1,\quad\eta_{h}\Big|_{\xi=0}=0,\quad|\nabla\eta_{h}|\leq\frac{C}{h\sqrt{\rho}}.
\end{align*}
We also denote $\varphi_{h}=\varphi\cdot\eta_{h}$ and it's not hard to show that
\begin{equation*}
    \xi^{2}|\nabla\varphi_{h}|^{2}\leq C\Big(\frac{\xi^{2}}{r^{2}}+\frac{\chi_{\{\xi\leq h\}}}{\rho}\Big)\leq\frac{C}{\rho}.
\end{equation*}
Now multiplying $\varphi_{h}^{2}w$ on both sides of \eqref{main again}, integration by parts implies
\begin{align*}
&\int_{B_{r}\setminus S}\xi^{2}A\Big(\nabla(\varphi_{h}w),\nabla(\varphi_{h}w)\Big)-\int_{B_{r}\setminus S}\xi^{2}w^{2}A(\nabla\varphi_{h},\nabla\varphi_{h})\\
=&\int_{B_{r}\setminus S}\xi^{2}\vec{f}\cdot\nabla(\varphi_{h}^{2}w)-\int_{B_{r}\setminus S}\xi^{2}g\varphi_{h}^{2}w.
\end{align*}
As $\lambda I\leq A\leq\Lambda I$ and $\ds\xi^{2}|\nabla\varphi_{h}|^{2}\leq\frac{C}{\rho}$, we have
\begin{align*}
&\int_{B_{r}\setminus S}\xi^{2}A\Big(\nabla(\varphi_{h}w),\nabla(\varphi_{h}w)\Big)\geq\lambda\int_{B_{r}\setminus S}\xi^{2}|\nabla(\varphi_{h}w)|^{2},\\
&\int_{B_{r}\setminus S}\xi^{2}w^{2}A(\nabla\varphi_{h},\nabla\varphi_{h})\leq C\cdot\Lambda\int_{B_{r}\setminus S}\frac{w^{2}}{\rho}.
\end{align*}
Besides, as $0\leq\varphi_{h}\leq1$,
\begin{align*}
|\int_{B_{r}\setminus S}\xi^{2}\vec{f}\cdot\nabla(\varphi_{h}^{2}w)|\leq&\varepsilon\int_{B_{r}\setminus S}\xi^{2}|\nabla(\varphi_{h}w)|^{2}+\frac{1}{\varepsilon}\int_{B_{r}\setminus S}\xi^{2}|\vec{f}|^{2}+C\int_{B_{r}\setminus S}\frac{w^{2}}{\rho},\\
|\int_{B_{r}\setminus S}\xi^{2}g\varphi_{h}^{2}w|\leq&\frac{1}{2}\int_{B_{r}\setminus S}\frac{w^{2}}{\rho}+\frac{1}{2}\int_{B_{r}\setminus S}\rho\xi^{4}g^{2}.
\end{align*}
Putting those estimates together yields that when $\varepsilon=\varepsilon(\lambda,\Lambda)$,
\begin{align*}
&\int_{B_{r/2}\cap\{\xi\geq h\}}\xi^{2}|\nabla w|^{2}\leq\int_{B_{r}\setminus S}\xi^{2}|\nabla(\varphi_{h}w)|^{2}\\
\leq&C\Big\{\int_{B_{r}\setminus S}(\xi^{2}|\vec{f}|^{2}+\rho\xi^{4}g^{2})+\int_{B_{r}\setminus S}\frac{w^{2}}{\rho}\Big\}.
\end{align*}
Finally, since in pointwise sense we have
\begin{equation*}
    \lim_{h\to0}\xi^{2}|\nabla w|^{2}\chi_{\{\xi\geq h\}}=\xi^{2}|\nabla w|^{2},
\end{equation*}
we can send $h\to0$, and apply Fatou lemma to finish the proof.
\end{proof}

Next, here is a global inequality for solutions of \eqref{main again} with zero boundary data.

\begin{proposition}\label{dirichlet}
Assume that $\lambda I\leq A\leq\Lambda I$. If $w\in H^{1}_{loc}$ is a solution of \eqref{main again} so that $w\Big|_{|x|\geq r}=0$, then there exists $C=C(\lambda,\Lambda)$ so that
\begin{equation}
    \int_{B_{r}\setminus S}\xi^{2}|\nabla w|^{2}dx\leq C\int_{B_{r}\setminus S}(\xi^{2}|\vec{f}|^{2}+\rho\xi^{4}g^{2})dx.
\end{equation}
\end{proposition}

\begin{proof}
The proof is similar to Proposition~\ref{caccioppoli}.
\end{proof}

\subsection{Weak solution}
Assume that $\lambda I\leq A\leq\Lambda I$. To find a solution of \eqref{main again} satisfying some boundary condition, we can use the Lax-Milgram lemma. Let the bi-linear form and linear functional
\begin{equation*}
B(w,\varphi):=\int_{B_{r}\setminus S}\xi^{2}\nabla\varphi^{T}A\nabla w dx,\quad F(\varphi)=\int_{B_{r}\setminus S}(\xi^{2}\vec{f}\cdot\nabla\varphi-\xi^{2}g\varphi)dx
\end{equation*}
be defined for $w,\varphi\in H^{1}_{0}(B_{r}\setminus S,\xi^{2}dx)$. If $w\in H^{1}_{0}(B_{r}\setminus S,\xi^{2}dx)$, i.e., $w=0$ when $|x|>r$, then by Proposition~\ref{thin poincare},
\begin{equation*}
    B(w,w)\geq\frac{\lambda}{5}\|w\|^{2}_{H^{1}_{0}(B_{r}\setminus S,\xi^{2}dx)},\quad B(w,v)\leq\Lambda\|w\|_{H^{1}_{0}(B_{r}\setminus S,\xi^{2}dx)}\|v\|_{H^{1}_{0}(B_{r}\setminus S,\xi^{2}dx)}.
\end{equation*}
Besides, the norm of $F(\varphi)$ is
\begin{equation*}
    |F(\varphi)|\leq(\|\vec{f}\|_{L^{2}(B_{r}\setminus S,\xi^{2}dx)}+\|g\|_{L^{2}(B_{r}\setminus S,\rho\xi^{4}dx)})\|\varphi\|_{H^{1}_{0}(B_{r}\setminus S,\xi^{2}dx)}.
\end{equation*}

In the special case $\vec{f}=0$ and $g=0$, we have the following existence result:
\begin{proposition}\label{exist}
    Assume that $\lambda I\leq A\leq\Lambda I$. If $w_{0}\in H^{1}(B_{r}\setminus S,\xi^{2}dx)$, then there exists a unique weak solution $w\in H^{1}(B_{r}\setminus S,\xi^{2}dx)$ of
    \begin{equation*}
        \mathrm{div}(\xi^{2}A\nabla w)=0
    \end{equation*}
    with boundary data $w_{0}$ (i.e., $w-w_{0}\in H^{1}_{0}(B_{r}\setminus S,\xi^{2}dx)$). Besides, there exists some $C=C(\lambda,\Lambda)$ such that
    \begin{equation}
        \int_{B_{r}\setminus S}\xi^{2}|\nabla w|^{2}dx\leq C\int_{B_{r}\setminus S}\xi^{2}|\nabla w_{0}|^{2}dx.
    \end{equation}
\end{proposition}
\begin{proof}
    It suffices to find a solution $w_{1}\in H^{1}_{0}(B_{r}\setminus S,\xi^{2}dx)$ so that $\ds w_{1}\Big|_{|x|>r}=0$ and
    \begin{equation*}
        \mathrm{div}(\xi^{2}A\nabla w_{1})=-\mathrm{div}(\xi^{2}A\nabla w_{0}).
    \end{equation*}
    Let $\vec{f}=-A\nabla w_{0}$, then as $w_{0}\in H^{1}(B_{r}\setminus S,\xi^{2}dx)$, the linear functional $F(\varphi)$ is bounded. By Lax-Milgram lemma, $w_{1}$ exists and
    \begin{equation*}
        \int_{B_{r}\setminus S}\xi^{2}|\nabla w_{1}|^{2}\leq C\int_{B_{r}\setminus S}\xi^{2}|\vec{f}|^{2}\leq C\int_{B_{r}\setminus S}\xi^{2}|\nabla w_{0}|^{2}.
    \end{equation*}
    If $w=w_{0}+w_{1}$, then $\mathrm{div}(\xi^{2}A\nabla w)=0$ with boundary data $w_{0}$.
\end{proof}

\section{Schauder estimate in average sense}\label{only average}
In this section, we give a $C^{1,\alpha}$ estimate in average sense of model equation \eqref{main again} at the origin. Before we move into the proof, we should mention that in Theorem~\ref{gammaschauder}, the assumption $A(0)=\delta^{ij}$ can be weakened, so that $\lambda I\leq A(0)\leq\Lambda I$, $A^{i,n+1}(0)=A^{n+1,i}(0)=0$ for $i\leq n$, and $A^{n,n}=A^{n+1,n+1}$. This can be reduced to the case $A(0)=\delta^{ij}$ by taking a linear coordinate change.

We first provide two approximation lemmas.
\begin{lemma}\label{constant coefficient}
If $h\in H^{1}(B_{1}\setminus S,\xi^{2}dx)$ is an $x_{n+1}$-even function solving the equation
\begin{equation*}
    \mathrm{div}(\xi^{2}\nabla h)=0,
\end{equation*}
then there exists a linear ``polynomial'' $l$ such that for all $r\leq1/2$,
\begin{equation}\label{eq. 3.1}
    \|l\|_{C^{0,1}(B_{1}\setminus S)}^{2}+\frac{1}{r^{n+4}}\int_{B_{r}\setminus S}\frac{|h-l|^{2}}{\rho}dx\leq C(\lambda,\Lambda)\int_{B_{1}\setminus S}\frac{h^{2}}{\rho}dx.
\end{equation}
\end{lemma}
\begin{proof}
Let $h_{k}\in H^{1}(B_{1}\setminus S,\xi^{2}dx)\cap C^{\infty}(B_{1}\setminus S)\cap C^{0}(\overline{B_{1}\setminus S})$ be a $H^{1}(B_{1}\setminus S,\xi^{2}dx)$ approximation of $h$, then $H_{k}=\xi h_{k}$ is a smooth $H^{1}(B_{1}\setminus S,dx)$ approximation of $H=\xi h$ and each $H_{k}$ vanishes at $S$. Therefore, $H\Big|_{B_{1}\cap S}=0$ in trace sense. Notice that
\begin{equation*}
    \Delta H=\Delta(\xi h)=\frac{1}{\xi}\mathrm{div}(\xi^{2}\nabla h)+h\Delta\xi=0,
\end{equation*}
then the desired inequality follows from the $C^{6}$ boundary estimate of $H$ in the $(\xi,\eta)$-coordinate, see Lemma~\ref{harmonic functions} and Remark~\ref{rmk. harmonic remark} below.
\end{proof}

\begin{lemma}\label{harmonic functions}
    Assume that $\Delta H(x)=0$ in $B_{1}\setminus S$ and $H\Big|_{S}=0$ in the trace sense. If $H\in H^{1}(B_{1},dx)$, then $H=H(x^{T},\xi,\eta)\in C^{\infty}(x^{T},\xi,\eta)$ near the origin, when written in the $(x^{T},\xi,\eta)$ coordinate, here $\xi\geq0$. See the definition of normal coordinates $(\xi,\eta)$ in \eqref{complex cordinate}. Moreover, we have $H=H(x^{T},\xi,\eta)\in C^{6}(x^{T},\xi,\eta)$ in $B_{1/2}\setminus S$ with
    \begin{equation*}
        H=P_{6}(x^{T},\xi,\eta)+o(|x^{T}|^{6}+\xi^{6}+\eta^{6}),
    \end{equation*}
    where $P_{6}$ is a sixth order polynomial in $(x^{T},\xi,\eta)$ with coefficients controlled by the $H^{1}(B_{1},dx)$ norm of $H$.
\end{lemma}
\begin{remark}\label{rmk. harmonic remark}
    Moreover, we can confirm that
    \begin{equation*}
        H/\xi\in span\{1,x_{1},\cdots,x_{n-1},x_{n},\rho\}\oplus O(|x|^{2})
    \end{equation*}
    by analyzing all polynomials in $(x^{T},\xi,\eta)$ whose degree is at most $6$. In particular, we can write
    \begin{equation*}
        h=H/\xi=l+O(|x|^{2})
    \end{equation*}
    for some linear ``polynomial'' (see Definition~\ref{def. "polynomials"}), so $|h-l|^{2}=O(|x|^{4})$, and then the estimate \eqref{eq. 3.1} follows immediately.
\end{remark}
The proof can be seen in \cite{DS15} (Theorem 4.5). We briefly sketch the main idea below.
\begin{proof}
    We can apply the weak Harnack principle to $H_{+}$ and $H_{-}$ on the slit $S$, showing that $H$ is $C^{\delta}$ at $B_{3/4}\cap S$ for some $\delta\in(0,1)$, hence $H\in C^{\delta}(B_{1/2})$ by an interior estimate. Now we can inductively show that $H$ is differentiable in $(\xi,\eta)$.
    \begin{itemize}
        \item[Step 1] We take discrete quotients in tangential directions to show that tangential derivatives $(D^{T})^{k}H$ are all $C^{\delta}$ for arbitrarily large $k$. This implies that $\Delta^{T}H+H_{\xi\xi}+H_{\eta\eta}=(1-4\xi^{2}-4\eta^{2})f$, where $f=\Delta^{T}H\in C^{\delta}$.
        \item[Step 2] As $H\Big|_{S}=0$, it also vanishes on $\{\xi=0\}$, so the boundary Schauder estimate applied to $H$ in $(\xi,\eta)$-coordinate implies that $H_{\xi},H_{\eta}\in C^{\delta}$.
        \item[Step 3] Let $\tilde{H}=\Delta^{T}H$, then it also satisfies $\Delta^{T}H+\tilde{H}_{\xi\xi}+\tilde{H}_{\eta\eta}=(1-4\xi^{2}-4\eta^{2})\tilde{f}$ for some $\tilde{f}$, we repeat Step 2 and obtain that $\tilde{H}_{\xi},\tilde{H}_{\eta}\in C^{\delta}$.
        \item[Step 4] Since $f_{\xi}=\tilde{H}_{\xi},f_{\eta}=\tilde{H}_{\eta}\in C^{\delta}$, we get that $H_{\xi\xi},H_{\xi\eta},H_{\eta\eta}\in C^{\delta}$ in $(\xi,\eta)$-coordinate.
        \item[Step 5] We repeat Step 2-4 infinitely many times, and finally obtain that $H$ is smooth in $(x^{T},\xi,\eta)$-coordinate.
    \end{itemize}
\end{proof}

\begin{lemma}\label{harmonic replacement}
Assume that $A^{ij}(0)=\delta^{ij}$ and $[A^{ij}]_{C^{\alpha}(B_{r}\setminus S)}\leq\varepsilon_{0}$. Let $w\in H^{1}_{loc}(B_{r}\setminus S,\xi^{2}dx)$ be a solution of \eqref{main again}, then there is a weak solution
\begin{equation}
    \mathrm{div}\Big(\xi^{2}\nabla h\Big)=0,\quad h\Big|_{(\partial B_{r/2})\setminus S}=w
\end{equation}
in the space $H^{1}(B_{r/2}\setminus S,\xi^{2}dx)$, and there exists $C$ so that
\begin{align}
    \int_{B_{r/2}\setminus S}\frac{h^{2}}{\rho}dx\leq&C\Big\{\int_{B_{r}\setminus S}(\xi^{2}|\vec{f}|^{2}+\rho\xi^{4}g^{2})dx+\int_{B_{r}\setminus S}\frac{w^{2}}{\rho}dx\Big\},\label{h1}\\
    \int_{B_{r/2}\setminus S}\frac{|h-w|^{2}}{\rho}dx\leq&C\Big\{\int_{B_{r}\setminus S}(\xi^{2}|\vec{f}|^{2}+\rho\xi^{4}g^{2})dx+\varepsilon_{0}^{2}r^{2\alpha}\int_{B_{r}\setminus S}\frac{w^{2}}{\rho}dx\Big\}.\label{h2}
\end{align}
\end{lemma}
\begin{proof}
We first use Proposition~\ref{caccioppoli} and~\ref{exist} to get the existence of $h$ with
\begin{equation*}
    \int_{B_{r/2}\setminus S}\xi^{2}|\nabla h|^{2}\leq C\int_{B_{r/2}\setminus S}\xi^{2}|\nabla w|^{2}\leq C\Big\{\int_{B_{r}\setminus S}(\xi^{2}|\vec{f}|^{2}+\rho\xi^{4}g^{2})+\int_{B_{r}\setminus S}\frac{w^{2}}{\rho}\Big\}.
\end{equation*}
If we denote $\ds\vec{f}'=\Big(A^{ij}(x)-\delta^{ij}\Big)\partial_{j}h\cdot \vec{e}_{i}$, then
\begin{align}\label{fprime}
    \int_{B_{r/2}\setminus S}\xi^{2}|\vec{f}'|^{2}\leq&C\varepsilon_{0}^{2}r^{2\alpha}\int_{B_{r/2}\setminus S}\xi^{2}|\nabla h|^{2}\notag\\
    \leq&C\varepsilon_{0}^{2}r^{2\alpha}\Big\{\int_{B_{r}\setminus S}(\xi^{2}|\vec{f}|^{2}+\rho\xi^{4}g^{2})+\int_{B_{r}\setminus S}\frac{w^{2}}{\rho}\Big\}
\end{align}
and $(h-w)$ satisfies $\ds \mathrm{div}\Big(\xi^{2}A(x)\cdot\nabla (h-w)\Big)=\mathrm{div}\Big(\xi^{2}(\vec{f}'-\vec{f})\Big)-\xi^{2}g$ with zero boundary data at $(\partial B_{r/2})\setminus S$. Proposition~\ref{thin poincare} and~\ref{dirichlet}, applied to $(h-w)$, give
\begin{equation*}
    \int_{B_{r/2}\setminus S}\frac{|h-w|^{2}}{\rho}\leq C\int_{B_{r/2}\setminus S}\xi^{2}|\nabla(h-w)|^{2}\leq C\int_{B_{r/2}\setminus S}\big(\xi^{2}|\vec{f}|^{2}+\xi^{2}|\vec{f}'|^{2}+\rho\xi^{4}g^{2}\big).
\end{equation*}
By using \eqref{fprime} we prove \eqref{h2}. Now \eqref{h1} follows from the triangle inequality.
\end{proof}
Now let us prove the Schauder estimate of $w$ in an average sense.

\begin{proof}[Proof of Proposition~\ref{gammaschauder}]Without loss of generality, we can assume $\ds\lim_{x\to0}\vec{f}=0$. Otherwise, if $\vec{f}_{0}=c_{1}e_{1}+,\cdots+c_{n}\vec{e}_{n}+c_{\rho}\nabla\rho\neq0$, then we study the equation of
\begin{equation*}
    w'=w-c_{1}x_{1}-\cdots-c_{n}x_{n}-c_{\rho}\rho,
\end{equation*}
which is reduced to the case $\ds\lim_{x\to0}\vec{f}=0$. We inductively define
\begin{equation}
    w_{0}=w,\quad w_{k+1}=w_{k}-l_{k}(k\geq0),\quad L_{k}=w-w_{k}=\sum_{i=0}^{k}l_{i},
\end{equation}
where linear ``polynomials'' satisfying $\mathrm{div}(\xi^{2}l_{k})=0$ will be chosen in \eqref{latereplace}. We also define
\begin{equation}
\vec{f}_{0}=\vec{f},\quad\vec{f}_{k+1}=\vec{f}_{k}+\Big(\delta^{ij}-A^{ij}(x)\Big)\Big(\partial_{j}l_{k}\Big)\vec{e}_{i},(k\geq0).
\end{equation}
By induction, we know $\ds\lim_{x\to0}\vec{f}_{k}=0$ and $\ds \mathrm{div}(\xi^{2}A\cdot\nabla w_{k})=\mathrm{div}(\xi^{2}\vec{f}_{k})+\xi^{2}g$ for all $k\geq0$. Let $\ds\varepsilon_{0}^{\frac{2}{n+4}}\leq\lambda\leq\frac{1}{4}$ be a shrinking rate, and we define three quantities
\begin{align}
    \sigma_{k}^{2}:=&\frac{1}{\lambda^{k(n+2+2\alpha)}}\int_{B_{\lambda^{k}}\setminus S}\frac{w_{k}^{2}}{\rho}dx,\\
    \phi_{k}:=&[\vec{f}_{k}]_{C^{\alpha}(B_{\lambda^{k}}\setminus S)},\\
    \gamma_{k}^{2}:=&\frac{1}{\lambda^{k(n+2+2\alpha)}}\int_{B_{\lambda^{k}}\setminus S}\rho\xi^{4}g^{2}dx\leq C(n,\alpha)\|g\|_{L^{p}(B_{1}\setminus S,\rho\xi^{2}dx)}^{2}.\label{g actual}
\end{align}
Here, $\sigma_{k}$ measures how $w$ and $L_{k}$ differ in $C^{1,\alpha}$ sense. Let $h_{k}\in H^{1}(B_{\lambda^{k}/2}\setminus S,\xi^{2}dx)$ be a replacement of $w_{k}$ in $B_{\lambda^{k}/2}\setminus S$, and $l_{k}$ be the linearization of $h_{k}$, i.e.
\begin{equation}\label{latereplace}
    \left\{
    \begin{aligned}
    &\mathrm{div}\Big(\xi^{2}\nabla h_{k}\Big)=0\\
    &h\Big|_{(\partial B_{\lambda^{k}/2})^{+}}=w_{k}
    \end{aligned}
    \right.,\quad l_{k}(x)=linearize(h_{k}).
\end{equation}
By applying Lemma~\ref{harmonic replacement} and Lemma~\ref{constant coefficient} to $w_{k+1}=(w_{k}-h_{k})+(h_{k}-l_{k})$, we have
\begin{equation}\label{iteration inequality}
\left\{
    \begin{aligned}
    &\phi_{k+1}\leq C(n)(\phi_{k}+\varepsilon_{0}\lambda^{k\alpha}\sigma_{k}+\varepsilon_{0}\lambda^{k\alpha}\gamma_{k}),\\
    &\sigma_{k+1}^{2}\leq C(n)\lambda^{2-2\alpha}\sigma_{k}^{2}+\frac{C(n)}{\lambda^{n+2+2\alpha}}(\phi_{k}^{2}+\gamma_{k}^{2}),\\
    &|l_{k}(0)|+|\nabla l_{k}(0)|\leq C(n)\lambda^{k\alpha}(\phi_{k}+\gamma_{k}+\sigma_{k}).
    \end{aligned}
    \right.
\end{equation}
Clearly, when $\varepsilon_{0}$ is very small, then there is room to choose a $\lambda=\lambda(n,\alpha)$ satisfying
\begin{equation*}
    \varepsilon_{0}^{\frac{2}{n+4}}\leq\lambda\leq\frac{1}{4},\quad C(n)\lambda^{2-2\alpha}\leq\frac{1}{2}.
\end{equation*}
With this the iteration inequality \eqref{iteration inequality} implies
\begin{align*}
    \sigma_{k}+\|L_{k}\|_{C^{0,1}(B_{1}\setminus S)}\leq&C(n,\alpha)(\phi_{0}+\sigma_{0}+\|g\|_{L^{p}(B_{1}\setminus S,\rho\xi^{2}dx)})\\
    \leq&C(\|\vec{f}\|_{C^{\alpha}(B_{1}\setminus S)}+\|g\|_{L^{p}(B_{1}\setminus S,\rho\xi^{2}dx)})+\|w\|_{L^{2}(B_{1}\setminus S,\xi^{2}dx)}).
\end{align*}
Besides, it follows that the sequence of linear ``polynomials'' $L_{k}$ converges to $L$ with
\begin{equation*}
    \|L\|_{C^{0,1}(B_{1}\setminus S)}\leq C(\|\vec{f}\|_{C^{\alpha}(B_{1}\setminus S)}+\|g\|_{L^{p}(B_{1}\setminus S,\rho\xi^{2}dx)})+\|w\|_{L^{2}(B_{1}\setminus S,\xi^{2}dx)}).
\end{equation*}
$L$ is a $C^{1,\alpha}$ approximation of $w$ because $\sigma_{k}$'s are bounded.

Finally, if $\ds\lim_{x\to0}\vec{f}\neq0$, then we can subtract a linear function from $w$ so that the remainder $w'$ satisfies $\mathrm{div}(\xi^{2}A\cdot\nabla w')=\mathrm{div}(\xi^{2}\vec{f}')$ with $\vec{f}'(0)=0$. Now we are reduced to the first case.
\end{proof}
\begin{remark}\label{absorb remark}
    If the right-hand side of \eqref{main again} has a term $\ds\frac{\xi}{\sqrt{\rho}}h$, where
    \begin{equation*}
        (h-c\sqrt{\rho}\frac{\partial\xi}{\partial x_{n}})=(h-c\frac{\xi}{2\sqrt{\rho}})\in C^{\alpha}(0)\cap O(|x|^{\alpha}),\quad(\mbox{i.e. }\frac{h-c\frac{\xi}{2\sqrt{\rho}}}{|x|^{\alpha}}\mbox{ is bounded}),
    \end{equation*}
    for some constant $c$, then it could be absorbed into $\mathrm{div}(\xi^{2}\vec{f})$ by setting
    \begin{equation}\label{eq. absorb construction}
        \vec{f}(x):=g(x^{T},\xi,\eta)\frac{\partial}{\partial\xi},\quad\mbox{where }\frac{\partial}{\partial\xi}=2\xi \vec{e}_{n}+2\eta \vec{e}_{n+1}.
    \end{equation}
    Here $g(x^{T},\xi,\eta)$ is expressed as the following integral:
    \begin{equation*}
        g(x^{T},\xi,\eta)=\frac{1}{(\xi^{2}+\eta^{2})\xi^{2}}\int_{0}^{\xi}t\sqrt{t^{2}+\eta^{2}}h(x^{T},t,\eta)dt.
    \end{equation*}
    In fact, we have
    \begin{align*}
        \mathrm{div}(\xi^{2}g\frac{\partial}{\partial\xi})=&\mathrm{div}(2\xi^{3}g\vec{e}_{n}+2\xi^{2}\eta g\vec{e}_{n+1})=\frac{\partial}{\partial x_{n}}(2\xi^{3}g)+\frac{\partial}{\partial x_{n+1}}(2\xi^{2}\eta g)\\
        =&\Big(\frac{\xi}{2\rho}\frac{\partial}{\partial\xi}(2\xi^{3}g)-\frac{\eta}{2\rho}\frac{\partial}{\partial\xi}(2\xi^{3}g)\Big)+\Big(\frac{\eta}{2\rho}\frac{\partial}{\partial\xi}(2\xi^{2}\eta g)+\frac{\xi}{2\rho}\frac{\partial}{\partial\xi}(2\xi^{2}\eta g)\Big)\\
        =&\xi^{2}\frac{\partial g}{\partial\xi}+\frac{4\xi^{3}+2\xi\eta^{2}}{\rho}g=\frac{1}{\rho}\frac{\partial}{\partial\xi}(\rho\xi^{2}g)=\frac{\xi}{\sqrt{\rho}}h(x^{T},\xi,\eta).
    \end{align*}
    Moreover, it could be verified that $\ds\int_{0}^{\xi}t\sqrt{t^{2}+\eta^{2}}dt$ is of order $\xi^{2}\sqrt{\rho}$, so if $h(x)$ is $L^{\infty}$ in the $x$-variable, then $\vec{f}(x)$ constructed in \eqref{eq. absorb construction} is $L^{\infty}$. Moreover, if $(h-c\frac{\xi}{2\sqrt{\rho}})\in C^{\alpha}(0)\cap O(|x|^{\alpha})$, by decomposing $h$ into $(h-c\frac{\xi}{2\sqrt{\rho}})$ and $c\frac{\xi}{2\sqrt{\rho}}$, we see that
    \begin{equation*}
        \vec{f}=c_{1}(\frac{\xi^{2}}{\rho}\vec{e}_{n}+\frac{\xi\eta}{\rho}\vec{e}_{n+1})+\vec{f}_{2},\quad\mbox{where }\vec{f}_{2}\in C^{\alpha}(0)\cap O(|x|^{\alpha}).
    \end{equation*}
    Notice that $\frac{\xi^{2}}{\rho}\vec{e}_{n}+\frac{\xi\eta}{\rho}\vec{e}_{n+1}$ is a linear combination of $\vec{e}_{n}$ and $\nabla\rho$, then we have that $\vec{f}\in C^{\alpha}(B_{1},0)$ as defined in Definition~\ref{def. 1.6}.
\end{remark}

\begin{remark}\label{xiphi on the RHS}
If $\alpha<1/2$, then the right-hand side of \eqref{main again} can also have a term $\xi\phi$, where $\phi\in L^{\infty}$. This is because when $g=\phi/\xi$, then $\xi\phi=\xi^{2}g$, and $g$ satisfies the weakened assumption \eqref{g weaken}.
\end{remark}

\section{Estimate of ratio \texorpdfstring{$v=u/\xi$}{Lg}}
\subsection{Equivalent descriptions of property \texorpdfstring{$(\mathcal{F})$}{Lg}}
The property $(\mathcal{F}_{A})$, or $(\mathcal{F})$ if the matrix $A$ is known, given in the introduction is easier for the readers to understand, but with it only it's hard to do estimate, so we provide a few parallel properties.

In $\mathbb{R}^{n+1}\setminus S$, we let $Cone_{r}(x^{T},0)$ (or simply $Cone_{r}(x^{T})$) be a cone of slope $1$, radius $r>0$, that is centered at $(x^{T},0)$. More explicitly,
\begin{equation}
    Cone_{r}(x^{T}):=\{y=(y^{T},y^{\perp}): |y^{\perp}|\leq r, |y^{T}-x^{T}|\leq|y^{\perp}|\}\setminus S.
\end{equation}
We say $f(x)$ defined in $B_{R}\setminus S$ satisfies properties $(\mathcal{F}_{1})$, $(\mathcal{F}_{2})$ or $(\mathcal{F}_{3})$ if:
\begin{itemize}
    \item[$(\mathcal{F}_{1})$] There exists a constant $C$, so that for every $Cone_{r}(x^{T})\in B_{R}\setminus S$,
    \begin{equation*}
        \Big\|\frac{f(y)}{\bar{u}_{A(x^{T})}(y)}\Big\|_{C^{\alpha}(Cone_{r}(x^{T})\setminus Cone_{r/2}(x^{T}))}\leq C,\quad\mbox{w.r.t. path distance}\ dist(\cdot,\cdot).
    \end{equation*}
    \item[$(\mathcal{F}_{2})$] There exists a constant $C$, so that for every $Cone_{r}(x^{T})\in B_{R}\setminus S$,
    \begin{equation*}
        \Big\|\frac{f(y)}{\bar{u}_{A(x^{T})}(y)}\Big\|_{C^{\alpha}(Cone_{r}(x^{T}))}\leq C,\quad\mbox{w.r.t. path distance}\ dist(\cdot,\cdot).
    \end{equation*}
    \item[$(\mathcal{F}_{3})$] There exists a constant $C$, so that for every $B_{2r}(x^{T})\subseteq B_{R}$, $\ds\Big|\frac{f(y)}{\bar{u}_{A(x^{T})}(y)}\Big|\leq C$ in $B_{r}(x^{T})$, and for every pair $y\in B_{r}(x^{T})$, $z\in Cone_{r}(x^{T})$, we have
    \begin{equation*}
        \Big|\frac{f(y)}{\bar{u}_{A(x^{T})}(y)}-\frac{f(z)}{\bar{u}_{A(x^{T})}(z)}\Big|\leq C\cdot dist(y,z)^{\alpha}.
    \end{equation*}
\end{itemize}
Again, $\bar{u}_{x^{T}}$ is an abbreviation of $\bar{u}_{A(x^{T})}$. Just like $(\mathcal{F})$, these three properties are also defined when assuming $A$ is known on $\mathbb{R}^{n-1}$.

If $\lambda I\leq A\leq\Lambda I$ and $[A]_{C^{\alpha}}$ is small, then the properties $(\mathcal{F})$, $(\mathcal{F}_{1})$, $(\mathcal{F}_{2})$, $(\mathcal{F}_{3})$ are equivalent up to a shrinking of radius, meaning that for example, if $f(x)$ has property $(\mathcal{F})$ in $B_{R}$, then it also has property $(\mathcal{F}_{1})$ in $B_{R/100}$. The proof uses some similar technique like in \cite{JV23} and is postponed to section~\ref{equivalence}.
\subsection{H\"older estimate} In this section, we study the H\"older continuity of \eqref{uniform}. We state the equation again,
\begin{equation}\label{uniform again}
    \mathrm{div}(A\nabla u)=\mathrm{div}(\frac{\vec{f}}{\sqrt{\rho}})\ \mbox{in}\ B_{1}\setminus S,\quad u\Big|_{S}=0.
\end{equation}
We start with the following Hardy inequality:
\begin{lemma}\label{lem. hardy}
    If $u\in C^{\infty}_{loc}(B_{r}/S)$ vanishes on the slit $S$, then for $v=u/\xi$, we have
\begin{equation*}
    \int_{B_{r}\setminus S}\frac{v^{2}}{\rho}dx=\int_{B_{r}\setminus S}\frac{u^{2}}{\xi^{2}\rho}dx\leq C\int_{B_{r}\setminus S}|\nabla u|^{2}dx.
\end{equation*}
\end{lemma}
\begin{proof}
Instead of using the $x$-coordinate, we use the $(x^{T},\xi,\eta)$-coordinate mentioned in \eqref{complex cordinate}. It follows that
\begin{align*}
    \int_{B_{r}\setminus S}|\nabla u|^{2}dx\geq&\int_{B_{r}\setminus S}|\nabla^{\perp}u|^{2}dx\geq\int_{\Omega}|\partial_{\xi}u|^{2}dx^{T}d\xi d\eta,\\
    \int_{B_{r}\setminus S}\frac{u^{2}}{\xi^{2}\rho}dx=&\int_{\Omega}\frac{u^{2}}{\xi^{2}}dx^{T}d\xi d\eta,
\end{align*}
where $\Omega=\{(x^{T},\xi,\eta):\xi>0,|x^{T}|^{2}+(\xi^{2}+\eta^{2})^{2}\leq r^{2}\}$ corresponds to the $B_{r}\setminus S$ region in $(x^{T},\xi,\eta)$-coordinate. We fix $x^{T}$ and $\eta$, and write $u(\xi)=u(x^{T},\xi,\eta)$, then it suffices to find a constant $C$ independent of $(x^{T},\eta)$ and the bound $A$, so that for all $u(\xi)\in H^{1}_{loc}(\mathbb{R}_{+})$ vanishing at $0$, we have
\begin{equation*}
C\int_{0}^{A}|u'|^{2}d\xi\geq\int_{0}^{A}\frac{u^{2}}{\xi^{2}}d\xi.
\end{equation*}
The strategy is to show a $W^{1,1}$ Hardy inequality $\ds\int_{0}^{A}\frac{|u'|}{\xi}d\xi\geq\int_{0}^{A}\frac{|u|}{\xi^{2}}d\xi$ and then replace $u$ with $u^{2}$, just like the method in \eqref{regular poincare}. For simplicity, we assume $w\in C^{\infty}$, then the $W^{1,1}$ Hardy inequality follows from changing the order of integration:
\begin{align*}
    \int_{0}^{A}\frac{|u(\xi)|}{\xi^{2}}d\xi\leq\int_{0}^{A}\frac{1}{\xi^{2}}\int_{0}^{\xi}|u'(h)|dhd\xi=\int_{0}^{A}\int_{h}^{A}\frac{1}{\xi^{2}}|u'(h)|d\xi dh\leq\int_{0}^{A}\frac{|u'(h)|}{h}dh.
\end{align*}
\end{proof}
Now let us prove the following H\"older estimate in the average sense:
\begin{lemma}\label{Holder average}
Assume that $A^{ij}(0)=\delta^{ij}$, $[A^{ij}]_{C^{\alpha}(B_{1}^{+})}\leq\varepsilon_{0}$, and $u\in L^{2}(B_{1}\setminus S,dx)\cap H^{1}_{loc}$ vanishing at $S$ is an $x_{n+1}$-even solution of \eqref{uniform again} with $\vec{f}(0)=0$. Given that $\varepsilon_{0}$ is small, there exist two constants $\bar{c}$ and $C=C(n,\alpha)$ (here, $C(n,\alpha)$ is universal while $\bar{c}$ depends on $u$) such that
\begin{equation}\label{average holder formula}
    |\bar{c}|^{2}+\frac{1}{r^{n+2\alpha}}\int_{B_{r}\setminus S}\frac{|u/\xi-\bar{c}|^{2}}{\rho}dx\leq C\Big\{\|\vec{f}\|_{C^{\alpha}(B_{1},0)}^{2}+\|u\|_{L^{2}(B_{1}\setminus S,dx)}^{2}\Big\}.
\end{equation}
Here, the norm $\|\vec{f}\|_{C^{\alpha}(B_{1},0)}$ is the same as that in Proposition~\ref{gammaschauder}.
\end{lemma}
\begin{proof}
    We write $v=u/\xi$, and also denote sequences $u_{k}=u-c_{k}\xi$ and $v_{k}=u_{k}/\xi$, where $c_{k}$'s are constants yet to be decided. It follows that $u_{k}$ satisfies the equation
    \begin{equation*}
        \mathrm{div}(A\nabla u_{k})=\mathrm{div}(\frac{\vec{f}_{k}}{\sqrt{\rho}}),\quad\vec{f}_{k}=\vec{f}-c_{k}\sqrt{\rho}(A-\delta^{ij})\nabla\xi.
    \end{equation*}
    As $\ds|\nabla\xi|=\frac{1}{2\sqrt{\rho}}$, we see that $|f_{k}(x)|\leq\phi_{k}|x|^{\alpha}$, where
    \begin{equation}\label{holder induction 1}
        \phi_{k}\leq C(\|\vec{f}\|_{C^{\alpha}(B_{1},0)}+\varepsilon_{0}|c_{k}|).
    \end{equation}

    As $u_{k}$ vanishes at $S$, we obtain the following estimate from Lemma~\ref{lem. hardy}:
\begin{equation}\label{hardy}
    \int_{B_{r}\setminus S}\frac{v_{k}^{2}}{\rho}dx=\int_{B_{r}\setminus S}\frac{u_{k}^{2}}{\xi^{2}\rho}dx\leq C\int_{B_{r}\setminus S}|\nabla u_{k}|^{2}dx.
\end{equation}
Besides, the Caccioppoli inequality for uniform elliptic equations implies that
\begin{equation}\label{standard caccioppoli}
    \int_{B_{r}\setminus S}|\nabla u_{k}|^{2}dx\leq C\int_{B_{2r}\setminus S}(\frac{u_{k}^{2}}{r^{2}}+\frac{|\vec{f}_{k}|^{2}}{\rho})dx.
\end{equation}
    
    We make the following claim:
    \begin{itemize}
        \item Claim: there exists a converging sequence $c_{k}$ and some $\lambda<1$, so that
    \begin{equation*}
        \int_{B_{\lambda^{k}}\setminus S}u_{k}^{2}dx\leq C\lambda^{k(n+2+2\alpha)},
    \end{equation*}
    \end{itemize}
    If the claim in correct, then we can use \eqref{hardy} and \eqref{standard caccioppoli} to get \eqref{average holder formula} for $\ds\bar{c}=\lim_{k\to\infty}c_{k}$.

    The remaining part is to prove the claim above. We follow a similar Campanato iteration like in Proposition~\ref{gammaschauder}. Without loss of generality, we assume that $\|\vec{f}\|_{C^{\alpha}(B_{1},0)}=\|u\|_{L^{2}(B_{1}\setminus S,dx)}=1$. We let $c_{0}=0$, $\ds\varepsilon_{0}^{\frac{2}{n+4}}\leq\lambda\leq1/4$, and define
\begin{equation*}
    \sigma_{k}^{2}=\frac{1}{\lambda^{k(n+2+2\alpha)}}\int_{B_{\lambda^{k}}\setminus S}u_{k}^{2}dx.
\end{equation*}
Clearly we have $\sigma_{0}=1$. By \eqref{standard caccioppoli} we have that
\begin{equation*}
    \int_{B_{\lambda^{k}/2}\setminus S}|\nabla u_{k}|^{2}\leq C\lambda^{k(n+2\alpha)}(\sigma_{k}^{2}+\phi_{k}^{2}),
\end{equation*}
where $\phi_{k}\leq C(1+\varepsilon_{0}|c_{k}|)$. Let $H_{k}$ be a harmonic replacement obtained using the Lax-Milgram lemma or energy minimizing method, so that it vanishes on $S$ in trace sense and
\begin{equation*}
    \Delta H_{k}=0\ \mbox{in}\ B_{\lambda^{k}/2}\setminus S,\quad H_{k}\Big|_{\partial B_{\lambda^{k}/2}\setminus S}=u_{k}.
\end{equation*}
By Lemma~\ref{harmonic functions} plus that $H_{k}$ vanishes on $S$, we can set $\ds c_{k+1}=c_{k}+\lim_{x\to0}\frac{H_{k}}{\xi}$, and
\begin{equation*}
    |c_{k}-c_{k+1}|^{2}+\frac{1}{\lambda^{n+4}r^{n+4}}\int_{B_{\lambda r}\setminus S}|H_{k}+c_{k}\xi-c_{k+1}\xi|^{2}\leq\frac{C}{r^{n+2}}\int_{B_{r}\setminus S}|H_{k}|^{2}
\end{equation*}
for all $\lambda<1$ if $H_{k}$ is $x_{n+1}$-even. It follows that
\begin{equation}\label{holder induction 2}
    |c_{k+1}|\leq|c_{k}|+C\lambda^{k\alpha}(\sigma_{k}+\phi_{k})
\end{equation}
As $u_{k+1}=(u_{k}-H_{k})+(H_{k}-c_{k+1}\xi+c_{k}\xi)$, we have
\begin{equation}\label{holder induction 3}
    \sigma_{k+1}^{2}\leq C(n)\lambda^{2-2\alpha}\sigma_{k}^{2}+\frac{C(n)}{\lambda^{n+2+2\alpha}}\phi_{k}^{2}.
\end{equation}
The iterative system \eqref{holder induction 1}\eqref{holder induction 2}\eqref{holder induction 3}, if we further assume $C(n)\lambda^{2-2\alpha}\leq1/2$, will imply the boundedness of $\sigma_{k}$ and convergence of $c_{k}$. Such a $\lambda$ which simultaneously satisfies $\ds\varepsilon_{0}^{\frac{2}{n+4}}\leq\lambda\leq1/4$ mentioned before exists, if $\varepsilon_{0}$ is small enough. This proves the claim.
\end{proof}

As a consequence of Lemma~\ref{Holder average}, we have the following interior regularity of $u/\xi$.
\begin{theorem}\label{Holder estimate}
Under the same assumption of Lemma~\ref{Holder average}, and further assuming that $\vec{f}$ is $C^{\alpha}$ inside $B_{1}$ and $\vec{f}(x)=0$ for all $x\in\mathbb{R}^{n-1}$, then $u(x)$ has the property $(\mathcal{F})$ in $B_{1/2}$, and $\sqrt{\rho}\nabla(u-c\xi)$ belongs to $C^{\alpha}(0)\cap O(|x|^{\alpha})$ in $B_{1/2}\setminus S$ (recall Remark~\ref{absorb remark}) for some constant $c$.
\end{theorem}
\begin{proof}
In each conic annulus $Cone_{r}\setminus Cone_{r/2}$, whose closure is contained in $B_{2r}\setminus B_{r/4}$, \eqref{uniform} is a non-degenerate equation, and we infer from \eqref{average holder formula} that
\begin{equation}\label{eq. 4.9}
    r^{-\frac{n+1}{2}}\|\tilde{u}\|_{L^{2}(Cone_{r}\setminus Cone_{r/2})}\leq r^{\frac{1}{2}+\alpha},\quad \tilde{u}=u-c\xi
\end{equation}
for some coefficient $c$. Then, by the standard boundary Schauder estimate, we have that under the distance function $dist(\cdot,\cdot)$,
    \begin{equation*}
        [u]_{C^{1,\alpha}(Cone_{r}\setminus Cone_{r/2})}\leq C r^{-1/2}
    \end{equation*}
    for some $C$ independent of $r$. Then, we see
    \begin{equation*}
        [\sqrt{\rho}\nabla\tilde{u}]_{C^{\alpha}(Cone_{r}\setminus Cone_{r/2})}\leq C,
    \end{equation*}
    so $[\sqrt{\rho}\nabla\tilde{u}]_{C^{\alpha}(Cone_{3/4})}\leq C$. Moreover, by \eqref{eq. 4.9}, we see that the limit of $\sqrt{\rho(x)}\nabla\tilde{u}(x)$ when $x$ tends to $0$ inside the cone $Cone_{1}$ is equal to $0$. For every $y\in B_{1/4}$, we construct
    \begin{equation*}
        z=y^{T}+\frac{|y^{T}|}{|y^{\perp}|}y^{\perp},
    \end{equation*}
    then similar to the computation in Lemma~\ref{sheaf},
    \begin{align}\label{sqrt rho nabla u}
        &\sqrt{\rho(y)}\nabla\tilde{u}(y)-\sqrt{\rho(z)}\nabla\tilde{u}(z)\\\notag
        =&\Big(\sqrt{\rho_{\kappa_{y^{T}}}(y)}\nabla\tilde{u}(y)-\sqrt{\rho_{\kappa_{y^{T}}}(z)}\nabla\tilde{u}(z)\Big)+\Big(\sqrt{\rho_{\kappa_{y^{T}}}(z)}-\sqrt{\rho(z)}\Big)\nabla\tilde{u}(z),
    \end{align}
    where
    \begin{equation*}
        \rho_{\kappa_{y^{T}}}=\sqrt{x_{n}^{2}+\kappa_{y^{T}}^{2}x_{n+1}^{2}},\quad\kappa_{y^{T}}=(\frac{\bar{a}(y^{T})^{nn}}{\bar{a}(y^{T})^{n+1,n+1}})^{1/2}.
    \end{equation*}
    In fact, since $z^{\perp}$ is parallel to $y^{\perp}$, the first term in \eqref{sqrt rho nabla u} is invariant if we add a $\frac{1}{2}$-order homogeneous function (for example, $c\xi$ or $c\xi_{\kappa_{y^{T}}}$) to $\tilde{u}$.
    
    Consequently, by combining \eqref{sqrt rho nabla u} with $[\sqrt{\rho}\nabla\tilde{u}]_{C^{\alpha}(Cone_{3/4})}\leq C$, we conclude that $\sqrt{\rho}\nabla\tilde{u}\in C^{\alpha}(0)\cap O(|x|^{\alpha})$.
    
    Besides, as $u\Big|_{S}=0$ and
    \begin{equation*}
        \xi\sim\frac{|x_{n+1}|}{\sqrt{\rho}}\quad\mbox{in }(Cone_{r}\setminus Cone_{r/2})\cap\{x_{n}<0\},
    \end{equation*}
    we conclude that $[u/\xi]_{C^{\alpha}(Cone_{r}\setminus Cone_{r/2})}\leq C$ with respect to $dist(\cdot,\cdot)$, which means $u$ has property $(\mathcal{F}_{1})$. As $(\mathcal{F}_{1})$ is ``equivalent'' to $(\mathcal{F})$ up to a shrinking of radius, see Appendix~\ref{equivalence}, we have finished the proof of Theorem~\ref{Holder estimate}.
\end{proof}

\subsection{Non-degeneracy}
We can use a similar method to prove that the growth rate of $u$ is exactly proportional to $\xi$ as long as $u$ is close to $\xi$ in $L^{2}$ sense.
\begin{proposition}\label{hopf}
    Assume that $A(0)=\delta^{ij}$ and $[A]_{C^{\alpha}(B_{1}\setminus S)}\leq\varepsilon_{0}$. $\vec{f}\Big|_{\mathbb{R}^{n-1}}=0$ and $[\vec{f}]_{C^{\alpha}(B_{1}\setminus S)}\leq\varepsilon_{0}$. $u$ is an $x_{n+1}$-even solution of \eqref{uniform again} with zero trace at $S$ and satisfies $\|u-\xi\|_{L^{2}(B_{1}\setminus S)}\leq\varepsilon_{0}$. Given that $\varepsilon_{0}$ is small enough, then we have
    \begin{equation}
        \inf_{B_{1/8}\setminus S}\frac{u}{\xi}\geq\frac{1}{2}.
    \end{equation}
\end{proposition}
\begin{proof}
    For every point $x^{T}\in\mathbb{R}^{n-1}\cap B_{1/2}$, as $|A(x^{T})-\delta^{ij}|\leq\varepsilon_{0}$, the corresponding homogeneous solution $\bar{u}_{x^{T}}=\bar{u}_{A(x^{T})}$ is close to $\xi$ in $L^{2}$ sense, so
    \begin{equation*}
        \|u-\bar{u}_{x^{T}}\|_{L^{2}(B_{1/2}(x^{T})\setminus S)}\leq\|u-\xi\|_{L^{2}(B_{1/2}(x^{T})\setminus S)}+\|\xi-\bar{u}_{x^{T}}\|_{L^{2}(B_{1/2}(x^{T})\setminus S)}\leq C\varepsilon_{0}.
    \end{equation*}
    At each center $x^{T}$, the function $\tilde{u}=(u-\bar{u}_{x^{T}})$ satisfies
    \begin{equation*}
        \mathrm{div}(A\nabla\tilde{u})=\mathrm{div}\Big(\frac{\vec{f}}{\sqrt{\rho}}+(\bar{A}-A)\nabla\bar{u}_{\kappa}\Big)=:\mathrm{div}(\frac{\vec{f}'}{\sqrt{\rho}}),
    \end{equation*}
    where $\vec{f}'$ vanishes at $x^{T}$ with $[\vec{f}']_{C^{\alpha}(B_{1/2}(x^{T})\setminus S)}\leq C\varepsilon_{0}$. We repeat the method in Theorem~\ref{Holder estimate} and get that
    \begin{equation*}
        \|\frac{u-\bar{u}_{x^{T}}}{\bar{u}_{x^{T}}}\|_{C^{\alpha}(Cone_{1/4}(x^{T})\setminus S)}\leq C\varepsilon_{0}.
    \end{equation*}
    For every $y\in B_{1/8}$, we construct
    \begin{equation*}
        z=y^{T}+\frac{|y^{T}|}{|y^{\perp}|}y^{\perp}
    \end{equation*}
    and get
    \begin{align*}
        |\frac{u-\xi}{\xi}(y)|\leq&|\frac{u-\xi}{\xi}(z)|+|\frac{u}{\xi}(z)-\frac{u}{\xi}(y)|\\
        =&|\frac{u-\xi}{\xi}(z)|+\frac{\bar{u}_{x^{T}}}{\xi}(y)|\frac{u}{\bar{u}_{x^{T}}}(z)-\frac{u}{\bar{u}_{x^{T}}}(y)|\\
        \leq&\|\frac{u-\xi}{\xi}\|_{C^{\alpha}(Cone_{1/4}\setminus S)}+C(A(x^{T}))\|\frac{u-\bar{u}_{x^{T}}}{\bar{u}_{x^{T}}}\|_{C^{\alpha}(Cone_{1/4}(x^{T})\setminus S)}\\
        \leq&C\varepsilon_{0}.
    \end{align*}
    When $\varepsilon_{0}$ is small, then $\ds|\frac{u-\xi}{\xi}(y)|\leq\frac{1}{2}$ for all $y\in B_{1/8}$, which proves the non-degeneracy assertion.
\end{proof}

\section{\texorpdfstring{$C^{1,\alpha}$}{Lg} boundary Harnack for straight boundary}
With Lemma~\ref{ratio}, plus the Theorem~\ref{Holder estimate} previously obtained, we can prove the $C^{1,\alpha}$ boundary Harnack principle for straight boundary.

\begin{proof}[Proof of Theorem~\ref{HOBH straight}]After a linear coordinate change, we can assume that $A(0)=\delta^{ij}$ and $\bar{u}_{A(0)}=\xi$. In fact, to say $u_{2}/u_{A(0)}\geq c_{0}>0$, this is independent of the homogeneous solution $\bar{u}_{A(0)}$, since the ratio of $\bar{u}_{A(0)}$'s between different $A(0)$'s is bounded from above and below.
\begin{itemize}
    \item Step 1: We first absorb $\phi_{i}$ into the divergence term by
    \begin{equation}\label{phi also absorb}
        \mathrm{div}(\vec{f}_{i})+\phi_{i}(x)=\mathrm{div}\Big(\vec{f}_{i}+\int_{0}^{x_{n+1}}\phi_{i}(x_{1},\cdots,x_{n},h)dh\cdot \vec{e}_{n+1}\Big)=:\mathrm{div}(\vec{f}_{i}').
    \end{equation}
    We write $\ds\vec{f}_{i}'=\frac{\sqrt{\rho}\vec{f}_{i}'}{\sqrt{\rho}}$ so that $\sqrt{\rho}\vec{f}_{i}'\in C^{\alpha}$ and vanishes at $\mathbb{R}^{n-1}$, then by Theorem~\ref{Holder estimate} we have that $u_{i}$'s have the property $(\mathcal{F})$ in $B_{3/4}$, and  $\sqrt{\rho}\nabla(u_{i}-c_{i}\xi)$ are $C^{\alpha}(0)\cap O(|x|^{\alpha})$ in $B_{3/4}$ for some constants $c_{i}$.
    \item Step 2: By Lemma~\ref{ratio}, we know $\ds w=\frac{u_{1}}{u_{2}}$ satisfies \eqref{eq. ratio equation}. Now we write
    \begin{equation*}
        \tilde{A}=(\frac{u_{2}}{\xi})^{2}A,\quad\vec{f}=\frac{u_{2}}{\xi}\frac{\vec{f}_{1}}{\xi}-\frac{u_{1}}{\xi}\frac{\vec{f}_{2}}{\xi},
    \end{equation*}
    and add two additional term mentioned in Remark~\ref{absorb remark} and Remark~\ref{xiphi on the RHS}:
    \begin{equation*}
        h=(\frac{\vec{f}_{2}}{\xi})\cdot(\sqrt{\rho}\nabla u_{1})-(\frac{\vec{f}_{1}}{\xi})\cdot(\sqrt{\rho}\nabla u_{2}),\quad \phi=\frac{u_{2}}{\xi}\phi_{1}-\frac{u_{1}}{\xi}\phi_{2}.
    \end{equation*}
    We then have that
    \begin{equation*}
        \mathrm{div}(\xi^{2}\tilde{A}\nabla w)=\mathrm{div}(\xi^{2}\vec{f})+\frac{\xi}{\sqrt{\rho}}h+\xi\phi.
    \end{equation*}
    Notice that as $\vec{f}_{i}\cdot\vec{e}_{n+1}=0$, we have that $(h-c\frac{\xi}{2\sqrt{\rho}})\in C^{\alpha}(0)\cap O(|x|^{\alpha})$ for some constant $c$, so by absorbing $\frac{\xi}{\sqrt{\rho}}h$ into $\mathrm{div}(\xi^{2}\vec{f})$ (see Remark~\ref{absorb remark}), we see that $w$ satisfies an equation in the form \eqref{main}.
    
    By the assumption of $\vec{f}_{i}$ and by the regularity of $u_{i}$, we know
    \begin{equation*}
        \tilde{A},\vec{f},h\in C^{\alpha}(0),\quad \phi\in L^{\infty}\mbox{ near the origin}.
    \end{equation*}
    Besides, by the assumption $u_{2}/\xi\geq c_{0}>0$, the matrix $\tilde{A}$ also satisfies the uniformly elliptic condition
    \begin{equation*}
        \tilde{\lambda}I\leq\tilde{A}\leq\tilde{\Lambda}I.
    \end{equation*}
    Therefore, $w$ at the origin is $C^{1,\alpha}$ in average sense by applying Proposition~\ref{gammaschauder}.
    \item Step 3: Such an argument works not only at the origin, but also for all points on $\mathbb{R}^{n-1}$. More precisely, assume that at $x^{T}\in\mathbb{R}^{n-1}$, $\ds\kappa=(\frac{\bar{a}^{nn}}{\bar{a}^{n+1,n+1}}(x^{T}))^{1/2}$, then we set
\begin{equation}
\rho_{\kappa}=\sqrt{x_{n}^{2}+\kappa^{2}x_{n+1}^{2}}.
\end{equation}
It turns out that $w$ is approximated by a linear ``polynomial'' $L$ in the variables $(x_{1},\cdots,x_{n},\rho_{\kappa})$, such that
\begin{equation}\label{eq. 5.3}
    \|L\|_{C^{0,1}(B_{1}(x^{T})\setminus S)}^{2}+\frac{1}{r^{n+2+2\alpha}}\int_{B_{r}(x^{T})\setminus S}\frac{|w-L|^{2}}{\rho}dx\leq C.
\end{equation}
By overlapping $B_{r}(x^{T})$ between different $x^{T}\in\mathbb{R}^{n-1}$, we obtain that $D^{T}w\in C^{\alpha}(B_{1/2}\cap\mathbb{R}^{n-1})$, so restriction of $w$ on $\mathbb{R}^{n-1}$ is $C^{1,\alpha}$ in the classical sense. Precisely speaking, let $L_{0}$ and $L_{y^{T}}$ be two linear ``polynomials'' in the variables $(x_{1},\cdots,x_{n},\rho)$ and $(x_{1},\cdots,x_{n},\rho_{\kappa_{y^{T}}})$, respectively. Let $r=|y^{T}|$, then by \eqref{eq. 5.3}, we have:
\begin{align*}
    &\frac{1}{r^{n+2+2\alpha}}\int_{B_{r}\cap B_{r}(y^{T})\setminus S}\frac{|w-L_{0}|^{2}}{\rho}dx\leq C,\\
    &\frac{1}{r^{n+2+2\alpha}}\int_{B_{r}\cap B_{r}(y^{T})\setminus S}\frac{|w-L_{y^{T}}|^{2}}{\rho}dx\leq C.
\end{align*}
These two estimates, together with the triangle equality, imply the $C^{0,1}$ smallness of $|L_{y^{T}}-L_{0}|$. In particular,
\begin{equation*}
    |D^{T}w(y^{T})-D^{T}w(0)|=|D^{T}L_{y^{T}}-D^{T}L_{0}|\leq C r^{\alpha}=C|y^{T}|^{\alpha}.
\end{equation*}
\end{itemize}
\end{proof}

\section{Proof of Theorem~\ref{application}}
In this section, to avoid ambiguity, we use $y$-coordinate to represent the coordinate before straightening the boundary, and $x$-coordinate to be the new coordinate. We re-write \eqref{ytox abuse} in a more clear way as
\begin{equation}\label{ytox}
    x_{i}=y_{i},\forall i\leq n-1,\quad x_{n}=y_{n}-\gamma(y^{T}),\quad x_{n+1}=y_{n+1}.
\end{equation}
The equation \eqref{U in y abuse} will be better written as
\begin{equation}\label{U in y}
    \partial_{y_{p}}(b^{pq}\partial_{y_{q}}U)=F(y).
\end{equation}
Here, we use the matrix $B(y)$ as the original matrix in $y$-coordinate satisfying the assumptions (a)-(d) mentioned in the introduction.

Assume that $U(y)\geq0$ solves the thin obstacle problem \eqref{U in y} whose blow-up is $C^{1,\alpha}$-converging to a $3/2$-homogeneous solution near the regular point $0$. After the coordinate change \eqref{ytox}, we have that $u_{m}=U_{y_{m}}$ with $1\leq m\leq n$ satisfy the equations:
\begin{equation}\label{equationu}
    \partial_{x_{i}}(a^{ij}\partial_{x_{j}}u_{m})+\partial_{x_{i}}(\frac{\partial x_{i}}{\partial y_{p}}\partial_{y_{m}}b^{pq}\cdot u_{q})=\partial_{y_{m}}F(y),\quad a^{ij}=b^{pq}\frac{\partial x_{i}}{\partial y_{p}}\frac{\partial x_{j}}{\partial y_{q}}.
\end{equation}
One can check that matrix $A$ also satisfies assumptions (a)-(d). We write
\begin{equation}
    \vec{f}_{m}=-\frac{\partial x_{i}}{\partial y_{p}}\partial_{y_{m}}b^{pq}\cdot u_{q}\vec{e}_{i},\quad\phi_{m}=\partial_{y_{m}}F(y(x)),
\end{equation}
then
\begin{equation}
    \mathrm{div}(A\nabla u_{m})=\mathrm{div}(\vec{f}_{m})+\phi_{m}.
\end{equation}
We also notice that the assumption $b^{n+1,i}=b^{i,n+1}=0$ for $i\leq n$ implies that $\vec{f}_{m}\cdot\vec{e}_{n+1}=0$ for all $m\leq n$. In fact, since
\begin{equation*}
    \frac{\partial x_{n+1}}{\partial y_{q}}=\delta_{n+1,p},\quad b^{n+1,q}=b^{n+1,n+1}\delta_{q,n+1},
\end{equation*}
and since $b^{n+1,n+1}$ is tangentially constant, we have
\begin{equation}\label{eq. additional explanation of fm being tangential}
    \vec{f}_{m}\cdot\vec{e}_{n+1}=-\frac{\partial x_{n+1}}{\partial y_{p}}\partial_{y_{m}}b^{pq}\cdot u_{q}=-\partial_{y_{m}}b^{n+1,n+1}\cdot u_{n+1}=0.
\end{equation}

Since the coordinate change \eqref{ytox} is $C^{1,\alpha}$ and $F(y)\in W^{1,\infty}$, we know $\phi_{m}\in L^{\infty}$ in the $x$-coordinate. Besides, once we know all $u_{m}$'s for $m\leq n$ have property $(\mathcal{F})$ (to be shown in Step 1 below), then so do $\vec{f}_{m}$'s.

Like in \eqref{phi also absorb}, we again absorb $\phi_{m}$ into the divergence term in Step 1\&2.
\begin{itemize}
    \item Step 1: Let $\vec{f}_{m}'$ be exactly the same as that in \eqref{phi also absorb} (so that $\phi_{m}$ is absorbed into the divergence term). By Theorem~\ref{Holder estimate}, as $\ds\sqrt{\rho}\vec{f}_{m}'$ are $C^{\alpha}(0)$ and vanish at $\mathbb{R}^{n-1}$, we have that $u_{m}$'s for $m\leq n$ are all $(\mathcal{F})$ in $B_{7/8}$.
    \item Step 2: We now show the non-degeneracy of $u_{m}$. Without loss of generality, we assume that $B(0)=\delta^{ij}$ and $\nabla\Gamma(0)=0$, then $a^{ij}(0)=\delta^{ij}$. As the blow-up of $U$ converges in $C^{1,\alpha}$ to the positive $3/2$ harmonic homogeneous solution $\ds Re\Big((x_{n}+ix_{n+1})^{3/2}\Big)$, we have that the blow-up of $u_{n}=\partial_{y_{n}}U$ converges to $\ds\xi=Re\Big((x_{n}+ix_{n+1})^{1/2}\Big)$ in $C^{\alpha}$ sense. As we have assumed that
    \begin{equation*}
        [A]_{C^{\alpha}},[\partial_{i}A]_{C^{\alpha}},\|\nabla F\|_{L^{\infty}}<\infty,
    \end{equation*}
    then by a suitable rescale, we can assume that
    \begin{equation*}
        [A]_{C^{\alpha}},[\partial_{i}A]_{C^{\alpha}},\|\nabla F\|_{L^{\infty}}\ll1.
    \end{equation*}
    When $\alpha<1/2$, we then have that $\sqrt{\rho}\vec{f}_{n}'$ vanish at $S$ and
\begin{equation}\label{eq. obstruction reason}
    [A]_{C^{\alpha}},\|u_{n}-\xi\|_{L^{2}},[\sqrt{\rho}\vec{f}_{n}']_{C^{\alpha}(0)}\leq\varepsilon_{0}
\end{equation}
By Proposition~\ref{hopf}, we see that $u_{n}/\xi\geq c_{0}>0$ in $B_{3/4}$.
\item Step 3: Finally, as discussed before, $u_{m}$'s being $(\mathcal{F})$ implies that $\vec{f}_{m}$'s for $m\leq n$ all have property $(\mathcal{F})$, so by Theorem~\ref{HOBH straight}, we see that $\ds w_{m}=\frac{u_{m}}{u_{n}}$ is $C^{1,\alpha}$ in average sense at the origin for $m\leq n$, and in particular, $w_{m}$ is $C^{1,\alpha}$ when restricted to $\mathbb{R}^{n-1}$ or $\Gamma$.
\end{itemize}
\begin{remark}\label{rmk. obstruction of alpha}
    Notice that \eqref{eq. obstruction reason} is the main reason why we assume that $\alpha<1/2$. The obstruction here is due to the fact that we don't have $\vec{f}_{n}'\equiv0$ on $\{x^{\perp}=0\}$.
\end{remark}

\appendix
\section{``Equivalence'' of regularity}\label{equivalence}
In the Appendix, we prove the ``equivalence'' of properties $(\mathcal{F})$, $(\mathcal{F}_{1})$, $(\mathcal{F}_{2})$, $(\mathcal{F}_{3})$. The most obvious relation is clearly $(\mathcal{F}_{3})\Rightarrow(\mathcal{F}_{2})\Leftrightarrow(\mathcal{F}_{1})$ up to a shrinking of radius.

The following lemmas prove the rest of ``equivalence''. In this section, when $\bar{u}_{\bar{A}}$ corresponds to the matrix $A(x^{T})$ where $x^{T}\in\mathbb{R}^{n-1}$, we can also simply write the homogeneous solution $\bar{u}_{A(x^{T})}$ as $\bar{u}_{x^{T}}$.
\begin{lemma}\label{sheaf}
    If $\lambda I\leq A\leq\Lambda I$, then $(\mathcal{F}_{2})\Rightarrow(\mathcal{F}_{3})$ up to a shrinking of radius.
\end{lemma}
\begin{proof}
    For simplicity, we assume that $x^{T}=0$, $A(0)=\delta^{ij}$, meaning that $\bar{u}_{0}=\xi$. If $y\in B_{r}(0)\setminus(S\cup Cone_{r})$, $z\in Cone_{r}(0)$, then $|y^{\perp}|\leq|y^{T}|$ and we let
    \begin{equation}\label{project to cone}
        w=y^{T}+\frac{|y^{T}|}{|y^{\perp}|}y^{\perp}\in Cone_{r}(0)\cap Cone_{r}(y^{T}).
    \end{equation}
    It's not hard to verify that
    \begin{equation*}
        |y-w|\leq2|y-z|,\quad|z-w|\leq3|y-z|.
    \end{equation*}
    In the cone $Cone_{r}(y^{T})$ we also have for $\bar{u}_{y^{T}}=\bar{u}_{A(y^{T})}$,
    \begin{equation*}
        C(\lambda,\Lambda)^{-1}\leq\bar{u}_{y^{T}}/\xi\leq C^{-1}(\lambda,\Lambda),\quad[f/\bar{u}_{y^{T}}]_{C^{\alpha}(Cone_{r})(y^{T})}\leq C.
    \end{equation*}
    Since $0,y^{T},y,w$ are co-planar, we have $\ds\frac{\bar{u}_{y^{T}}}{\xi}(w)=\frac{\bar{u}_{y^{T}}}{\xi}(y)$. Therefore,
    \begin{align*}
        |\frac{f(y)}{\xi(y)}-\frac{f(z)}{\xi(z)}|\leq&|\frac{f(y)}{\bar{u}_{y^{T}}(y)}\frac{\bar{u}_{y^{T}}(y)}{\xi(y)}-\frac{f(w)}{\bar{u}_{y^{T}}(w)}\frac{\bar{u}_{y^{T}}(w)}{\xi(w)}|+|\frac{f(w)}{\xi(w)}-\frac{f(z)}{\xi(z)}|\\
        \leq&\frac{\bar{u}_{y^{T}}}{\xi}(y)|\frac{f}{\bar{u}_{y^{T}}}(y)-\frac{f}{\bar{u}_{y^{T}}}(w)|+|\frac{f}{\xi}(w)-\frac{f}{\xi}(z)|\\
        \leq&C(|y-w|^{\alpha}+|w-z|^{\alpha})\leq(2^{\alpha}+3^{\alpha})C|y-z|^{\alpha}.
    \end{align*}
\end{proof}
\begin{lemma}
    If $\lambda I\leq A\leq\Lambda I$ and $[A]_{C^{\alpha}}$ is small, then $(\mathcal{F})\Rightarrow(\mathcal{F}_{1})$ up to a shrinking of radius.
\end{lemma}
\begin{proof}
    For simplicity, we assume that $x^{T}=0$, $A(0)=\delta^{ij}$, meaning that $\bar{u}_{0}=\xi$. It suffices to show that for every $y,z\in Cone_{r}(0)\setminus Cone_{r/2}(0)$, we have
    \begin{equation*}
        |\frac{\bar{u}_{y^{T}}(y)}{\xi(y)}h(y)-\frac{\bar{u}_{z^{T}}(z)}{\xi(z)}h(z)|\leq C\cdot dist(y,z)^{\alpha}.
    \end{equation*}
    The left-hand-side is a sum of $\ds D_{1}=\frac{\bar{u}_{z^{T}}(z)}{\xi(z)}\cdot|h(y)-h(z)|$ and
    \begin{equation*}
        D_{2}=|h(y)|\cdot|\frac{\bar{u}_{y^{T}}(z)}{\xi(z)}-\frac{\bar{u}_{z^{T}}(z)}{\xi(z)}|,\quad D_{3}=|h(y)|\cdot|\frac{\bar{u}_{y^{T}}(y)}{\xi(y)}-\frac{\bar{u}_{y^{T}}(z)}{\xi(z)}|.
    \end{equation*}
    As $h(x)$ is $C^{\alpha}$, we have $D_{1}\leq C\cdot dist(y,z)^{\alpha}$. When $z$ is fixed, $\ds\frac{\bar{u}_{\bar{A}}(z)}{\xi(z)}$ is Lipschitz about $\bar{A}$, which is a $C^{\alpha}$ function as $[A]_{C^{\alpha}}$ is small. Therefore, we also have
    \begin{equation*}
        D_{2}\leq C\cdot dist(y^{T},z^{T})^{\alpha}\leq C\cdot dist(y,z)^{\alpha}.
    \end{equation*}
    To estimate $D_{3}$, we first use complex coordinate to write
    \begin{equation*}
        y_{n}+iy_{n+1}=r_{y}e^{i\theta_{y}},\quad z_{n}+iz_{n+1}=r_{z}e^{i\theta_{z}},\quad\theta_{y},\theta_{z}\in(-\pi,\pi).
    \end{equation*}
    It follows that $\bar{u}_{y^{T}}/\xi$ is determined and Lipschitz by the argument angle $\theta_{y},\theta_{z}$. Besides, as $A\in C^{\alpha}$, the Lipschitz norm about $\theta$ is in fact bounded by
    \begin{equation*}
        \Big|\partial_{\theta}\frac{\bar{u}_{y^{T}}(r,\theta)}{\xi(r,\theta)}\Big|\leq C|y^{T}|^{\alpha}
    \end{equation*}
    The angle between $y^{\perp},z^{\perp}$ without crossing the slit $S$, say $\angle{(y^{\perp},z^{\perp})}=|\theta_{y}-\theta_{z}|$, which takes value in $(0,2\pi)$, satisfies
    \begin{equation*}
        \angle{(y^{\perp},z^{\perp})}\leq\left\{\begin{aligned}
            \sin^{-1}\frac{dist(y^{\perp},z^{\perp})}{|y^{\perp}|}&,\ \mbox{if}\ dist(y^{\perp},z^{\perp})<|y^{\perp}|\\
            2\pi&,\ \mbox{if}\ dist(y^{\perp},z^{\perp})\geq|y^{\perp}|
        \end{aligned}\right.\leq C\frac{dist(y^{\perp},z^{\perp})}{|y^{\perp}|}.
    \end{equation*}
    Notice that when $y\in Cone_{r}$, $|y^{\perp}|\geq|y^{T}|$, so
    \begin{align*}
        D_{3}\leq&C\cdot\Big\|\partial_{\theta}\frac{\bar{u}_{y^{T}}(r,\theta)}{\xi(r,\theta)}\Big\|_{L^{\infty}}\cdot|\theta_{y}-\theta_{z}|\leq C|y^{T}|^{\alpha}\cdot\frac{dist(y^{\perp},z^{\perp})}{|y^{\perp}|}\\
        \leq&C\Big(\frac{dist(y^{\perp},z^{\perp})}{|y^{\perp}|}\Big)^{1-\alpha}dist(y^{\perp},z^{\perp})^{\alpha}\leq C3^{1-\alpha}dist(y,z)^{\alpha}.
    \end{align*}
    Here we used $2|y^{\perp}|\geq|z^{\perp}|$ in the last step, since $y,z\in Cone_{r}(0)\setminus Cone_{r/2}(0)$.
\end{proof}
\begin{lemma}
    If $\lambda I\leq A\leq\Lambda I$ and $[A]_{C^{\alpha}}$ is small, then $(\mathcal{F}_{3})\Rightarrow(\mathcal{F})$ up to a shrinking of radius.
\end{lemma}
\begin{proof}
    We write $\ds h(x)=\frac{f(x)}{\bar{u}_{x^{T}}(x)}$ in $B_{R}$, then it suffices to show $h(x)\in C^{\alpha}(B_{R/100})$. Let $x,y\in B_{R/100}$, then as $x\in Cone_{R/10}(x^{T})$, we have that
    \begin{equation*}
        |h(x)-\frac{f(y)}{\bar{u}_{x^{T}}(y)}|=|\frac{f(x)}{\bar{u}_{x^{T}}(x)}-\frac{f(y)}{\bar{u}_{x^{T}}(y)}|\leq C|x-y|^{\alpha}.
    \end{equation*}
    Therefore, as it is obvious that $h\in L^{\infty}(B_{R/100})$, we have
    \begin{align*}
        |h(x)-h(y)|\leq&C|x-y|^{\alpha}+|h(y)|\cdot|\frac{\bar{u}_{y^{T}}(y)}{\bar{u}_{x^{T}}(y)}-1|\\
        \leq&C|x-y|^{\alpha}+C|x^{T}-y^{T}|^{\alpha}.
    \end{align*}
\end{proof}


\end{document}